\documentclass[11pt,letterpaper]{amsart}
\usepackage[utf8]{inputenc}
\usepackage[T1]{fontenc}
\usepackage{amsmath,amsthm,amssymb}
\usepackage[shortlabels]{enumitem}
\usepackage{xcolor}
\usepackage{bbm}
\usepackage{todonotes}
\usepackage{mathtools}

\DeclarePairedDelimiter{\floor}{\lfloor}{\rfloor}

\newtheorem{theorem}{Theorem}[section]
\newtheorem{remark}[theorem]{Remark}
\newtheorem{definition}[theorem]{Definition}

\newtheorem*{conjecture*}{Conjecture}

\newtheorem*{acknowledgement}{Acknowledgements}

\newtheorem{lemma}[theorem]{Lemma}
\newtheorem{prop}[theorem]{Proposition}

\newcommand{\R}{\mathbb{R}}
\newcommand{\T}{\mathbb{T}}

\newcommand{\Z}{\mathbb{Z}}
\newcommand{\gesim}{\gtrsim}
\newcommand{\lesim}{\lesssim}
\newcommand{\eps}{\varepsilon}

\newcommand{\cH}{\mathcal{H}}

\newcommand{\norm}[1]{\| #1 \|}

\newcommand{\RapDec}{\mathrm{RapDec}}

\newcommand{\hphi}{\widehat{\phi}}
\newcommand{\hmu}{\widehat{\mu}}
\newcommand{\hsigma}{\widehat{\sigma}}
\newcommand{\Bad}{\mathrm{Bad}}

\newcommand{\cJ}{\mathcal{J}}

\newcommand{\spt}{\mathrm{spt}\,}

\newcommand{\cpsi}{\Check{\psi}}
\newcommand{\cE}{\mathcal{E}}
\newcommand{\Tjbad}{\T_{j,\mathrm{bad}}}
\newcommand{\Csep}{C_{\text{sep}}}

\newcommand{\dist}{\mathrm{dist}}
\newcommand{\surf}{\mathrm{Surf}}
\newcommand{\one}{\mathbbm{1}}

\begin{document}

\title[Falconer distance set problem]{New improvement to Falconer distance set problem in higher dimensions}
\author[X. Du, Y. Ou, K. Ren and R. Zhang]{Xiumin Du, Yumeng Ou, Kevin Ren and Ruixiang Zhang}

\begin{abstract}
We show that if a compact set $E\subset \mathbb{R}^d$ has Hausdorff dimension larger than $\frac{d}{2}+\frac{1}{4}-\frac{1}{8d+4}$, where $d\geq 3$, then there is a point $x\in E$ such that the pinned distance set $\Delta_x(E)$ has positive Lebesgue measure. This improves upon bounds of Du--Zhang and Du--Iosevich--Ou--Wang--Zhang in all dimensions $d \ge 3$. We also prove lower bounds for Hausdorff dimension of pinned distance sets when $\dim_H (E) \in (\frac{d}{2} - \frac{1}{4} - \frac{3}{8d+4}, \frac{d}{2}+\frac{1}{4}-\frac{1}{8d+4})$, which improves upon bounds of Harris and Wang--Zheng in dimensions $d \ge 3$.
\end{abstract}

\maketitle

\section{Introduction}
\setcounter{equation}0
A classical question in geometric measure theory, introduced by Falconer in the early 80s (\cite{falconer1985hausdorff}) is, how large does the Hausdorff dimension of a compact subset of ${\R}^d$, $d\ge 2$ need to be to ensure that the Lebesgue measure of the set of its pairwise Euclidean distances is positive.

Let $E\subset\mathbb{R}^d$ be a compact set, its \emph{distance set} $\Delta(E)$ is defined by
$$
\Delta(E):=\{|x-y|:x,y\in E\}\,.
$$

\begin{conjecture*}\label{conj} \textup{[Falconer]}
Let $d\geq 2$ and $E\subset\mathbb{R}^d$ be a compact set. Then
$$
{\dim_H}(E)> \frac d 2 \Rightarrow |\Delta(E)|>0.
$$
Here $|\cdot|$ denotes the Lebesgue measure and ${\dim_H}(\cdot)$ is the Hausdorff dimension.
\end{conjecture*}

The main result in this paper improves the best-known dimensional threshold towards the Falconer conjecture in dimensions $d\geq 3$.

\begin{theorem}\label{main}
Let $d\geq 3$ and $E\subset\mathbb{R}^d$ be a compact set. Then
$$
{\dim_H}(E)> \frac d 2+\frac{1}{4}-\frac{1}{8d+4} \Rightarrow |\Delta(E)|>0.
$$
\end{theorem}

Falconer's conjecture remains open in all dimensions as of today. It has attracted a great amount of attention in the past decades. To name a few landmarks: in 1985, Falconer \cite{falconer1985hausdorff} showed that $|\Delta(E)|>0$ if ${\dim_H}(E)>\frac{d}{2}+\frac{1}{2}$. Bourgain \cite{bourgain1994distance} was the first to lower the threshold $\frac{d}{2}+\frac{1}{2}$ in dimensions $d=2, d=3$ and to use the theory of Fourier restriction in the Falconer problem. The thresholds were further improved by Wolff \cite{wolff1999decay} to $\frac{4}{3}$ in the case $d=2$, and by Erdo\u{g}an \cite{erdogan2005} to $\frac{d}{2}+\frac{1}{3}$ when $d\geq 3$. These records were only very recently rewritten:
\[
\begin{cases}
\frac{5}{4}, &d=2, \quad\qquad\text{(Guth--Iosevich--Ou--Wang \cite{guth2020falconer})}\\
\frac{9}{5}, &d=3, \quad \qquad \text{(Du--Guth--Ou--Wang--Wilson--Zhang \cite{DGOWWZ})}\\
\frac d2+\frac 14+\frac{1}{8d-4}, &d\geq 3, \quad \qquad \text{(Du--Zhang \cite{du2019sharp})}\\
\frac{d}{2}+\frac 14, &d\geq 4 \text{ even}, \quad \text{(Du--Iosevich--Ou--Wang--Zhang \cite{du2021improved})}.
\end{cases}
\]
Our main result in this paper further improves the thresholds in all dimensions $d\geq 3$. Now, the gap between the best-known threshold and the conjectured one is $\frac 14 -\frac{1}{8d+4}$ when $d\geq 3$. This is the first time that the gap gets smaller than $\frac{1}{4}$ and the first time that the gap in higher dimensions is smaller than that in dimension $d=2$.
Similar to \cite{guth2020falconer, du2021improved}, we in fact prove a slightly stronger version of the main theorem regarding the pinned distance set.

\begin{theorem}\label{thm: pinned}
Let $d\geq 3$ and $E\subset\mathbb{R}^d$ be a compact set. Suppose that $\dim_H (E)>\frac{d}{2}+\frac{1}{4}-\frac{1}{8d+4}$. Then there is a point $x\in E$ such that the pinned distance set $\Delta_x(E)$ has positive Lebesgue measure, where
$$
\Delta_x(E):=\{|x-y|:\, y\in E\}.
$$
\end{theorem}

We also get a Hausdorff dimension version if $\dim_H(E)$ is lower than the threshold. Let $f(\alpha) = \alpha \cdot \frac{2d+1}{d+1} - (d-1)$. Then the following theorem gives a non-trivial result for $\dim_H (E) \in (\frac{d^2-1}{2d+1}, \frac{d^2+d}{2d+1}]$.
\begin{theorem}\label{thm:distance_hausdorff}
    Let $d\ge 2$, $0 < \alpha \le d-1$, and $E\subset\mathbb{R}^d$ be a compact set. Suppose that $\dim_H (E) = \alpha$. Then for any $\eps > 0$,
    \begin{equation*}
        \dim_H (\{ x \in E : \dim_H (\Delta_x (E)) \le \min(f(\alpha), 1)-\eps\}) < \alpha.
    \end{equation*}
    In particular, $\sup_{x \in E} \dim_H (\Delta_x (E)) \ge \min(f(\alpha), 1)$, and if $\cH^\alpha(E)>0$, then $\dim_H (\Delta_x (E)) \ge \min(f(\alpha), 1)$ for $\cH^\alpha$-almost all $x \in E$.
\end{theorem}

In dimension $d=2$, a lot of recent progress has been made. Liu \cite{liu2020hausdorff} showed that if $\dim_H (E) = \alpha>1$, then $\sup_{x \in E} \dim_H (\Delta_x (E)) \ge \min(\frac{4\alpha}{3}-\frac 23, 1)$. For small values of $\alpha>1$, Shmerkin \cite{shmerkin2020dimensions} improved this bound by using entropy, and later, Stull \cite{stull2022pinned} (see also \cite{fiedler2023dimension}) made a further improvement, he showed using algorithmic complexity theory that  $\dim_H (\Delta_x (E)) \ge \frac{\alpha}{4}+\frac 12$ for many $x$. Theorem \ref{thm:distance_hausdorff} deals with the case $\dim_H (E) \leq 1$.

When $d\geq 3, \dim_H (E) = \alpha > \frac d2$, the previous best-known results are as follows: Harris \cite{harris2021low} proved that $\sup_{x \in E} \dim_H (\Delta_x (E)) \geq \min(\frac{2d-1}{d} \alpha - (d-1), 1)$; for even dimensions $d\geq 4$, Wang--Zheng \cite{wang2022improvement} improved this result to $\sup_{x \in E} \dim_H (\Delta_x (E)) \geq \min(\frac{2d}{d+1} \alpha - \frac{2d^2 - d - 2}{2(d+1)}, 1)$. Theorem \ref{thm:distance_hausdorff} further improves upon these two results.

However, getting a lower bound for $\dim_H (\Delta_x (E))$ given $\dim_H (E) = \frac{d}{2}$ is much more challenging. Falconer in \cite{falconer1985hausdorff} proved that $\dim_H (\Delta(E)) \ge \frac{1}{2}$, and a pinned version $\sup_{x \in E} \dim_H (\Delta_x (E)) \ge \frac{1}{2}$ was proved in \cite{oberlin2014spherical}. The reason $\frac{1}{2}$ is a natural barrier is that if there existed a $\frac{1}{2}$-dimensional ring $R \subset [1, 2]$, then the distance set of $R^d$ is contained in $\sqrt{R}$, so it has Hausdorff dimension $\frac{1}{2}$. Even though the works of Bourgain \cite{bourgain2003erdos} and Katz--Tao \cite{katz2001some} showed a quantitative discretized sum-product theorem that rules out the existence of $R$, it is still a challenging problem to obtain explicit bounds for the discretized sum-product problem, which asks for the largest exponent $\gamma > 0$ such that $\min(|A + A|_\delta, |A \cdot A|_\delta) \gesim |A|^{1+\gamma}$ for any Katz--Tao $(\delta, \frac{1}{2})$-set $A \subset \R^1$ (see \cite{katz2001some} for the definition of such sets). See \cite{guth2021discretized}, \cite{orponen2023projections}, \cite{mathe2023discretised}, and \cite{ren2023furstenberg} for the best known bounds for the discretized sum-product problem.

For the Falconer distance set problem, the only explicit improvements known over $\frac{1}{2}$ for $\dim_H (\Delta(E))$ when $\dim_H (E) = \frac{d}{2}$ were derived in \cite{shmerkin2022non} and \cite{shmerkin2021distance} for $d = 2, d=3$, and the latter paper also proved box dimension results for $d \ge 4$. In personal communication, Shmerkin--Wang extended Stull's bound in $d = 2$ to the case $\dim_H (E) = 1$. Furthermore, by plugging in the sharp radial projection estimates of \cite{orponen2022kaufman} into the proofs of Theorems 1.2 and 1.3 of \cite{shmerkin2021distance}, we get improved bounds over those stated in these theorems for $d \ge 3$. In summary, the previously known results are:
\begin{itemize}
    \item if $d = 2$, $\dim_H (E) = 1$, then $\sup_{x \in E} \dim_H (\Delta_x (E)) \ge \frac{3}{4}$;

    \item if $d = 3$, $\dim_H (E) = \frac{3}{2}$, then $\sup_{x \in E} \dim_H (\Delta_x (E)) \ge \frac{5}{8}$;

    \item if $d \ge 4$, $\dim_H (E) = \frac{d}{2}$, then $\sup_{x \in E} \dim_B (\Delta_x (E)) \ge \frac{d+2}{2(d+1)}$.
\end{itemize}
A key obstruction to Hausdorff dimension estimates in \cite{shmerkin2021distance} when $d \ge 4$ is that $\frac{d}{2}$-dimensional measures don't necessarily have decay around small neighborhoods of $(d-1)$-planes when $d \ge 4$, so one cannot apply Bourgain's discretized projection theorem to such measures.

Note that in Theorem \ref{thm:distance_hausdorff}, $f(\frac{d}{2}) = \frac{d+2}{2(d+1)} > \frac{1}{2}$. This matches the best-known bound $\frac 58$ when $d = 3$ using an entirely different approach. And when $d\geq 4$, this is the first time one obtains an explicit improved bound over $\frac 12$ for the Hausdorff dimension of pinned distance sets of a $\frac{d}{2}$-dimensional set, (and this matches the box dimension bound as in the above).

Finally, let us remark what is known when $\alpha < \frac{d}{2}$. Falconer \cite{falconer1985hausdorff} proved that $\dim_H (\Delta(E)) \ge \min(\alpha - \frac{d-1}{2}, 0)$; this was improved in dimensions $d = 2$ and $3$ to $\sup_{x \in E} \dim_H (\Delta_x (E)) \ge \frac{\alpha + 1}{d + 1}$ for $\alpha \in (\frac{d-1}{2}, \frac{d}{2})$ by \cite{orponen2022kaufman}, \cite{shmerkin2021distance}. In dimension $d \ge 4$ however, their approach only recovers box dimension estimates. Our bound $\sup_{x \in E} \dim_H (\Delta_x (E)) \ge f(\alpha)$ from Theorem \ref{thm:distance_hausdorff} is weaker than $\frac{\alpha + 1}{d + 1}$ for all $\alpha < \frac{d}{2}$, but it works for Hausdorff dimension for all $d \ge 3$.

\subsection*{Old and new ideas}

We will extend and refine the good tube/bad tube and decoupling method pioneered by \cite{guth2020falconer} for dimension $d = 2$ and continued in \cite{du2021improved} for even dimensions $d$. To state our contributions, we find it helpful to first recall the good tube/bad tube decomposition in \cite{guth2020falconer}, \cite{du2021improved}. Fix $\alpha > \frac{d}{2}$ and a set $E$ with $\dim_H (E) > \alpha$; thus, we can find two $\alpha$-dimensional Frostman measures $\mu_1, \mu_2$ supported on separated subsets of $E$. Then we fix a single scale $r$, and consider the set of $r$-tubes (i.e. $r$-neighborhoods of lines) that intersect $E$. We say that an $r$-tube $T$ is good if $\mu_2 (T) \lesim r^{\frac{d}{2}-\eps}$, and bad otherwise. The paper \cite{guth2020falconer} developed a decoupling framework that can handle good tubes very well, so it remains to show that there are very few bad tubes. When $d = 2$, this was solved by appealing to the $n=2$ case of Orponen's radial projection theorem, which roughly states that given a Frostman measure $\mu$ in $\R^n$ with dimension $> n-1$, most tubes satisfy $\mu_2 (T) \lesim r^{n-1}$. When $d > 2$ is even, we can no longer directly apply the radial projection theorem because $\alpha$ is usually $\le d-1$; to overcome this \cite{du2021improved} came up with a partial fix by first projecting $\mu_2$ onto a generic $(\frac{d}{2}+1)$-dimensional subspace of $\R^d$ (assuming $d$ is even), so the projected measure $\mu_{2*}$ still remains $\alpha$-dimensional, and then applying Orponen's radial projection theorem to $\mu_{2*}$ in $\R^{\frac{d}{2}+1}$. Despite the seemingly wasteful application of the initial orthogonal projection, the threshold $r^{\frac{d}{2}-\eps}$ is sharp up to $\eps$: the example of a $(\frac{d}{2}+1)$-hyperplane shows that $r^{\frac{d}{2}}$ is the best threshold one can get in general (and we get an even worse threshold $r^{\frac{d-1}{2}}$ for odd dimensions via an $(\frac{d+1}{2})$-hyperplane example; see e.g. Section 5.2 of \cite{du2021improved}).

    To improve upon \cite{du2021improved}, we must find a way to deal with this hyperplane example. As it so happens, we are elated if this situation occurs: if our set $E$ is contained in a lower-dimensional hyperplane $H = \R^n$, then we can simply work in $\R^n$ where existing distance set results are already better than what we need. Thus, it would be nice to establish a dichotomy: either we can improve upon the barrier $r^{\lfloor \frac{d}{2} \rfloor}$, or it turns out that much of $E$ is contained in a lower-dimensional hyperplane. Fortunately, this wish is indeed true and is one of our main innovations of the paper. More precisely, we establish a framework, for the first time in the context of the distance problems, that analyzes how much a fractal measure concentrates near a low dimensional plane hence allows us to reduce the problem to the aforementioned dichotomy. We show that either (1) the measure is dictated by wave packets (i.e. microlocalized pieces of the measure) that display good transversality. In this case, desirable multilinear estimates or projection results can apply to deduce an improved threshold $r^{\alpha-\epsilon}$ (a significant improvement over the barrier $r^{\lfloor \frac{d}{2} \rfloor}$ and would lead to gains in the decoupling step), or (2) the problem can be reduced to lower dimensions hence existing distance set results can apply. 
    
On the technical level, this is achieved by a careful analysis of the interaction between the plates of intermediate dimensions (i.e. neighborhood of planes) and wave packets. Our starting point is a new radial projection theorem of the third author \cite{KevinRadialProj}, which we precisely state as Theorem \ref{conj:threshold}. Roughly, it states that most $r$-tubes either have $\mu_2$-mass $\lesim r^{\alpha-\eps}$ or lie in some heavy $(r^\kappa, \lceil \alpha \rceil)$-plate, which is the $r^{\kappa}$-neighborhood of a $\lceil \alpha \rceil$-dimensional hyperplane with $(\mu_1 + \mu_2)$-mass $\gesim r^{\eta}$ (here, $\eta \ll \kappa \ll \eps$). The main obstacle that we need to deal with then becomes the heavy $(r^\kappa, \lceil \alpha \rceil)$-plates, and a key novelty of this paper is an upgrade of the good/bad framework that analyzes the interaction between these plates and the wave packets.

Note that this strategy in fact indicates a key difference between the distance problem in two dimensions and that in higher dimensions: in dimension three and higher, it is possible to study the distance set using information from the distance problem in lower dimensions. This is apparently not possible in the plane, in which case the measure is assumed to have Hausdorff dimension greater than one hence cannot concentrate on a line. As far as we are aware of, it seems that such a special feature of the distance problem in the higher dimensional setting has not been explored in the literature before.

Here are some more details, where we fix $d = 3$ for concreteness. Following the setup in \cite{du2021improved}, we let $\mu_1, \mu_2$ be $\alpha$-dimensional Frostman measures supported on subsets $E_1, E_2 \subset E$ with $\text{dist}(E_1, E_2) \gesim 1$. Let $R_0$ be a large number and define the scales $R_j = 2^j R_0$. Fix a scale $j$; we shall work with $R_j^{-1/2+\beta}$-tubes. Let $G(x)$ to be the set of $y \in E_1$ such that $x, y$ don't both lie in some heavy $(R_j^{-\kappa}, 2)$-plate. 
Define a good $R_j^{-1/2+\beta}$-tube to be one with mass $\lesim R_j^{-\alpha/2+\eps}$ and doesn't lie in any heavy plate, and $\mu_{1,g}$ to be a part of $\mu_1$ from all good tubes. With these definitions, we can try to follow the framework of \cite{du2021improved} and show that $\mu_1|_{G(x)}$ and $\mu_{1,g}$ are close in $L^1$-norm. There are two subtleties with this approach. First, the error in $\norm{\mu_1|_{G(x)} - \mu_{1,g}}_{L^1}$ may contain contributions from tubes $T$ at the border of $G(x)$, i.e. $T \subset G(x)$ but $2T \not\subset G(x)$. To overcome this issue, we introduce a probabilistic wiggle. If we work with heavy $(aR_j^{-\kappa}, 2)$-plates for a uniformly random $a \in [1, 2]$, then the borderline error will be small on average, so there exists a choice for $a$ to make the borderline error small.

The second issue is that we do not have the luxury of working at a single scale, so we need to introduce some ideas from \cite[Appendix B]{shmerkin2022non}. Specifically, we will construct a decreasing sequence $\R^3 \supset G_0 (x) \supset G_1 (x) \supset \cdots$ such that $\R^3 \setminus G_0 (x)$ is contained in a $(R_0^{-\kappa/2}, 2)$-plate, $\mu(G_0 (x) \setminus G_\infty (x)) \le R_0^{-\kappa/2}$, and $G_j (x)$ is disjoint from all heavy $(aR_j^{-\kappa},2)$-plates containing $x$ (where $a \in [1, 2]$ is the probabilistic wiggle). Since all the $G_j$'s are close in measure, we can combine multiscale information to show that $\mu_1|_{G_0 (x)}$ and $\mu_{1,g}$ are close in $L^1$-norm.

The final step is to show $G_0 (x)$ is large. This is equivalent to showing that any $(R_0^{-\kappa/2}, 2)$-plate has small ($\le \frac{1}{2}$) measure for $R_0$ large enough. If this is not true, then by compactness we can find a hyperplane with nonzero measure, and then we reduce to Falconer in 2D. In this case, we have strong bounds for distance sets (e.g. see Wolff \cite{wolff1999decay}) because $\dim_H (E) > 1.5$ is large.

\begin{remark}
    In dimensions $d \ge 4$, a slightly simpler and more intuitive choice of good measures exists. We refer the interested reader to  \cite{DOKZFalconerDec} for more details, where such a direction is pursued and same results in Theorems \ref{main} and \ref{thm: pinned} (except for the $d=3$ case) are obtained as a consequence of new weighted refined decoupling estimates. In dimension $d=3$ though, the strategy still works but doesn't seem to yield as good a result as the current paper.
\end{remark}

    We point out that, in some sense, the dimensional threshold we obtained already fully exploited the power of the methods of this paper and the companion paper \cite{DOKZFalconerDec}. The radial projection theorem of \cite{KevinRadialProj} is sharp in the following sense: let $E_1, E_2$ each be unions of $r^{-\alpha}$ many $r$-balls satisfying an $\alpha$-dimensional spacing condition such that $\dist (E_1, E_2) \ge \frac{1}{2}$, and let $\mu_1, \mu_2$ be probability measures supported on $E_1, E_2$ respectively such that each $r$-ball in $E_i$ has $\mu_i$-measure $r^\alpha$, $i = 1, 2$. Then, we know that many $r$-tubes can intersect at least one $r$-ball in $E_1$ and one $r$-ball in $E_2$. Thus, many $r$-tubes can have both $\mu_1$ and $\mu_2$-mass at least $r^\alpha$, so $r^\alpha$ is the best possible threshold for the good tubes in this paper. To make further progress, we suggest to look at improving the decoupling framework.

    On the other hand, the threshold $r^\alpha$ is already stronger than all the previously best known thresholds for good tubes in all dimensions $d\geq 2$ (see \cite{guth2020falconer, du2021improved}). For comparison, in two dimensions, it was shown that one can reduce to good $r$-tubes whose measure is at most $\sim r$, which is weaker than $r^\alpha$. The possibility of obtaining such a better threshold for the good tubes underscores another key difference between the distance problem in higher dimensions and that in the plane. Indeed, one can easily see that for an evenly distributed probablity measure in the plane, the threshold $r$ would be the best possible.

\subsection*{Outline of the paper}

In Section \ref{sec-mainest}, we outline two main estimates and prove Theorem \ref{thm: pinned} using them. In Section \ref{sec-radial}, we list several results that we will use to bound the bad part, including the new radial projection estimate by the third author \cite{KevinRadialProj} and two geometric lemmas governing the physical location of small plates with large mass. In Section \ref{sec:prop2.1}, we construct the good measure and prove the first main estimate - Proposition \ref{mainest1}. In Section \ref{sec:prop2.2}, we prove the second main estimate - Proposition \ref{mainest2} using refined decoupling. In Section \ref{sec:hausdorff}, we prove Theorem \ref{thm:distance_hausdorff} using the two main estimates and a framework of Liu \cite{liu2020hausdorff}. In Section \ref{sec:othernorms}, we give some remarks about the extension of Theorem \ref{thm: pinned} to more general norms and its connection with the Erd\H os distance problem.
\vspace{.1in}

\subsection*{Notations}
Throughout the article, we write $A\lesssim B$ if $A\leq CB$ for some absolute constant $C$; $A\sim B$ if $A\lesssim B$ and $B\lesssim A$; $A\lesssim_\eps B$ if $A\leq C_\eps B$; $A\lessapprox B$ if $A\leq C_\eps R^\eps B$ for any $\eps>0$, $R>1$.

For a large parameter $R$, ${\rm RapDec}(R)$ denotes those quantities that are bounded by a huge (absolute) negative power of $R$, i.e. ${\rm RapDec}(R) \leq C_N R^{-N}$ for arbitrarily large $N>0$. Such quantities are negligible in our argument.

For $x\in \R^d$ and $t>0$, $B(x, t)$ is the ball in $\R^d$ of radius $t$ centered at $x$. A $\delta$-tube is the intersection of an infinite cylinder of radius $\delta$ and $B(0, 10)$. This is not standard (usually, a tube is defined to be a finite $\delta$-cylinder with length $1$), but this definition won't cause problems and is slightly more convenient for us.

For a set $E\subset \R^d $, let $E^c = \R^d \setminus E$, and $E^{(\delta)}$ the $\delta$-neighborhood of $E$.
For subsets $E_1, E_2 \subset \R^d$, $\dist (E_1,E_2)$ is their Euclidean distance.

For $A \subset X \times Y$ and $x \in X$, define the slice $A|_x = \{ y \in Y : (x, y) \in A \}$. Similar definition for $A|_y$, when $y\in Y$.

For a measure $\mu$ and a measurable set $G$, define the restricted measure $\mu|_G$ by $\mu|_G (A) = \mu(G \cap A)$ for all measurable $A\subset \R^d$. 

We say a measure $\mu$ supported in $\R^d$ is an $\alpha$-dimensional measure with constant $C_\mu$ if it is a probability measure satisfying that
\begin{equation*}
\mu(B(x,t)) \leq C_\mu t^\alpha,\qquad \forall x\in \mathbb{R}^d,\, \forall t>0.
\end{equation*}

An $(r,k)$-plate $H$ in $\R^d$ is the $r$-neighborhood of a $k$-dimensional affine plane in the ball $B^d(0,10)$. More precisely, 
$$
H:=\{z\in B(0,10): \dist(z, P_H) < r\},
$$
where $P_H$ is a $k$-dimensional affine plane, which is called the central plane of $H$. We can also write $H= P_H^{(r)} \cap B(0, 10)$. A $C$-scaling of $H$ is 
$$
CH=\{z\in B(0,10): \dist(z, P_H) < Cr\}.
$$
Denote the surface of $H$ by
$$
\surf(H) := \{z\in B(0, 10): \dist(z, P_H) = r\}\,.
$$
Since a $\delta$-tube is a $(\delta, 1)$-plate, the same conventions apply to tubes.

To prevent the reader from feeling too attached to $B(0, 10)$, we will see in the proofs that $B(0, 10)$ can be replaced by any smooth bounded convex body that contains $B(0, 10)$.

We say that an $(r,k)$-plate $H$ is $\gamma$-concentrated on $\mu$ if $\mu(H) \ge \gamma$.

Given a collection $\cH$ of plates in $\R^d$ and a point $x\in \R^d$, define $\cH(x):=\{H\in\cH:x\in aH\}$,  where $a$ is a ``master parameter'' that will be defined later.

We will work with a collection $\cE_{r,k}$ of essentially distinct $(r,k)$-plates with the following properties:
\begin{itemize}
    \item Each $(\frac{r}{2}, k)$-plate intersecting $B(0, 1)$ lies in at least one plate of $\cE_{r,k}$;

    \item For $s \ge r$, every $(s, k)$-plate contains $\lesim_{k,d} \left( \frac{s}{r} \right)^{(k+1)(d-k)}$ many $(r, k)$-plates of $\cE_{r,k}$.
\end{itemize}
For example, when $k = 1$ and $d = 2$, we can simply pick $\sim r^{-1}$ many $r$-tubes in each of an $r$-net of directions. This generalizes to higher $k$ and $d$ via a standard $r$-net argument, which can be found in \cite[Section 2.2]{KevinRadialProj}.

\begin{acknowledgement}
XD is supported by NSF DMS-2107729 (transferred from DMS-1856475), NSF DMS-2237349 and Sloan Research Fellowship. YO is supported by NSF DMS-2142221 and NSF DMS-2055008. KR is supported by a NSF GRFP fellowship. RZ is supported by NSF DMS-2207281 (transferred from DMS-1856541),  NSF DMS-2143989 and the Sloan Research Fellowship.
\end{acknowledgement}

\section{Main estimates} \label{sec-mainest}
\setcounter{equation}0
In this section, we outline two main estimates, from which Theorems \ref{thm: pinned} and \ref{thm:distance_hausdorff} follow.

Let $E\subset \mathbb{R}^d$ be a compact set with positive $\alpha$-dimensional Hausdorff measure. Without loss of generality, assume that $E$ is contained in the unit ball, and there are subsets $E_1$ and $E_2$ of $E$, each with positive $\alpha$-dimensional Hausdorff measure, 
and $\dist(E_1, E_2)\gtrsim 1$. Then there exist $\alpha$-dimensional probability measures $\mu_1$ and $\mu_2$ supported on $E_1$ and $E_2$, respectively.

To relate the measures to the distance set, we consider their pushforward measures under the distance map. For a fixed point $x\in E_2$, let $d^x:E_1\to \R$ be the pinned distance map given by $d^x(y):=|x-y|$. Then, the pushforward measure $d^x_\ast(\mu_1)$, defined as
\[
\int_{\mathbb{R}} \psi(t)\,d^x_\ast(\mu_1)(t)=\int_{E_1}\psi(|x-y|)\,d\mu_1(y),
\]
is a natural measure that is supported on $\Delta_x(E_1)$.

In the following, we will construct another complex-valued measure $\mu_{1,g}^x$ that is the \emph{good} part of $\mu_1$ with respect to $\mu_2$ depending on $x$, and study its pushforward under the map $d^x$. The main estimates are the following.

\begin{prop} \label{mainest1} Let $d\geq 2$, $k \in \{ 1, 2, \cdots, d-1 \}$, $k-1<\alpha\leq k$, and $\eps > 0$. Then there exists a small $\beta (\eps) > 0$ such that the following holds for sufficiently large $R_0 (\beta, \eps)$. Assume $\mu_{1,g}^x$ has been constructed following the procedure in Section \ref{sec: good} below. Then there is a subset $E_2' \subset E_2$ with $\mu_2(E_2') \ge 1 - R_0^{-\beta}$ and for each $x \in E_2'$, there exists a set $G(x) \subset B^d(0,10)$ where $B^d(0,10) \setminus G(x)$ is contained within some $(R_0^{-\beta}, k)$-plate, such that the following estimate holds:
\begin{equation} \label{eqn:mu1good1}
        \norm{d_*^x (\mu_1|_{G(x)}) - d_*^x (\mu_{1,g}^x)}_{L^1} \le R_0^{-\beta}.
\end{equation}
\end{prop}

\begin{prop} \label{mainest2} Let $d\geq 2$, $0 < \alpha < d$, and $\eps > 0$. Then for sufficiently small $\beta(\eps) \in (0, \eps)$ in the construction of $\mu_{1,g}^x$ in Section \ref{sec: good} below,
$$	
\int_{E_2}  \| d^x_*(\mu_{1,g}^x) \|_{L^2}^2 d \mu_2(x) \lesim_{d,\alpha,\eps} \int_{\mathbb{R}^d} |\xi|^{-\frac{\alpha d}{d+1}+\eps} |\hmu_1(\xi) |^2 \,d \xi + R_0^d.
$$
\end{prop}

\begin{remark}
    Broadly speaking, $G(x)$ removes the contributions from small plates that contain large mass. This makes it possible to efficiently apply the new radial projection theorem, Theorem \ref{conj:threshold}.
\end{remark}




\begin{proof} [Proof of Theorem \ref{thm: pinned} using Propositions \ref{mainest1} and \ref{mainest2}]
For $d\geq 3$ and $\frac{d}{2}<\alpha < \frac{d+1}{2}$, let $k=\floor{\frac d 2} +1$ so that $k-1<\alpha\leq k$. If $\mu_1$ gives nonzero mass to some $k$-dimensional affine plane, then we are done by applying \cite{falconer1985hausdorff} to that $k$-plane since $\alpha>\frac d2 \geq \frac{k+1}{2}$. Thus, assume $\mu_1$ gives zero mass to every $k$-dimensional affine plane. By a compactness argument, there exists $r_0 > 0$ such that $\mu_1 (H) < \frac{1}{1000}$ for any $(r_0, k)$-plate $H$. Now the two propositions tell us that there is a point $x \in E_2$ such that
\begin{gather*}
    \norm{d_*^x (\mu_1|_{G(x)}) - d_*^x (\mu_{1,g}^x)}_{L^1} \le \frac{1}{1000}, \\
    \norm{d_*^x (\mu_{1,g}^x)}^2_{L^2} \lesim I_\lambda (\mu_1) + R_0^d < \infty,
\end{gather*}
by choosing $R_0$ sufficiently large. Here $I_{\lambda} (\mu_1)$ is the $\lambda$-dimensional energy of $\mu_1$ and $\lambda=d-\frac{\alpha d}{d+1}+\eps$, by a Fourier representation for $I_\lambda$:
\[
I_\lambda (\mu) = \int \int |x-y|^{-\lambda} d\mu(x) d\mu(y)=C_{d, \lambda}  \int_{\mathbb{R}^d} |\xi|^{\lambda-d} |\hmu(\xi) |^2 \,d \xi.
\]
One has $I_\lambda(\mu_1)<\infty$ if $\lambda<\alpha$, which is equivalent to $\alpha>\frac{d(d+1)}{2d+1}=\frac{d}{2}+\frac{1}{4}-\frac{1}{8d+4}$.
Now $B^d(0,10) \setminus G(x)$ is contained in some $(R_0^{-\beta}, k)$-plate $H$. If $R_0$ is chosen sufficiently large such that $R_0^\beta > r_0^{-1}$, then
\begin{equation*}
    \mu_1 (G(x)) \ge 1 - \mu_1 (H) > 1 - \frac{1}{1000}.
\end{equation*}
Since $d_*^x (\mu_1|_{G(x)})$ is a positive measure, its $L^1$ norm is $\mu_1 (G(x)) > 1 - \frac{1}{1000}$. Thus,
    \begin{multline*}
        \int_{\Delta_x (E)} |d_*^x (\mu_{1,g}^x)| = \int |d_*^x (\mu_{1,g}^x)| - \int_{\Delta_x (E)^c} |d_*^x (\mu_{1,g}^x)| \\
        \ge 1 - \frac{2}{1000} - \int |d_*^x (\mu_1|_{G(x)}) - d_*^x (\mu_{1,g}^x)| \ge 1 - \frac{3}{1000}.
    \end{multline*}
On the other hand,
\begin{equation*}
    \int_{\Delta_x (E)} |d_*^x (\mu_{1,g}^x)| \le |\Delta_x (E)|^{1/2} \left( \int |d_*^x (\mu_{1,g}^x)|^2 \right)^{1/2}.
\end{equation*}
Therefore, $|\Delta_x (E)| > 0$.
\end{proof}

The proof of Theorem \ref{thm:distance_hausdorff} is deferred until Section \ref{sec:hausdorff}.

\section{Radial projections and heavy plates} \label{sec-radial}
In this section, we list several results that we will use in Sections \ref{sec:pruning} and \ref{sec:bad measure} to bound the bad part of $\mu_1$.

We'll use the following new radial projection estimate, which follows from \cite[Theorem 1.13]{KevinRadialProj}.

\begin{theorem}\label{conj:threshold}
Let $d\geq 2$, $k \in \{ 1, 2, \cdots, d-1 \}$, $k-1 < \alpha \le k$, and fix $\eta, \eps > 0$, and two $\alpha$-dimensional measures $\mu_1, \mu_2$ with constants $C_{\mu_1}, C_{\mu_2}$ supported on $E_1, E_2 \subset B(0, 1)$ respectively. There exists $\gamma > 0$ depending on $\eta, \eps, \alpha, k$ such that the following holds. Fix $\delta < r < 1$. Let $A$ be the set of pairs $(x, y) \in E_1 \times E_2$ satisfying that $x$ and $y$ lie in some $\delta^\eta$-concentrated $(r,k)$-plate on $\mu_1 + \mu_2$. Then there exists a set $B \subset E_1 \times E_2$ with $\mu_1 \times \mu_2 (B) \le \delta^{\gamma}$ such that for every $x \in E_1$ and $\delta$-tube $T$ through $x$, we have
\begin{equation*}
    \mu_2 (T \setminus (A|_x \cup B|_x)) \lesim \frac{\delta^\alpha}{r^{\alpha-(k-1)}} \delta^{-\eps}.
\end{equation*}
The implicit constant may depend on $\eta, \eps, \alpha, k, C_{\mu_1}, C_{\mu_2}$.
\end{theorem}

\begin{remark}
    (a) The roles of $\mu_1$ and $\mu_2$ in Theorem \ref{conj:threshold} are interchangeable, so the conclusion also holds for $\mu_1$ instead of $\mu_2$.

    (b) If $\alpha>d-1$, then the numerology of Theorem \ref{conj:threshold} doesn't apply. Instead, Orponen's radial projection theorem \cite{orponen2018radial} in dimension $d$ applies. The result (stated in \cite[Lemma 3.6]{guth2020falconer} for $d = 2$, but can be generalized to all dimensions $d$) is that for $\gamma = \eps/C$, there exists a set $B \subset E_1 \times E_2$ with $\mu_1 \times \mu_2 (B) \le \delta^{\gamma}$ such that for every $x \in E_1$ and $\delta$-tube $T$ through $x$, we have
\begin{equation*}
    \mu_2 (T \setminus B|_x) \lesim \delta^{d-1-\eps}.
\end{equation*}
Note that the set $A$ of ``concentrated pairs'' is not needed here.
    
    (c) If $r \sim \delta$, we can obtain a slightly better result by projecting to a $k$-dimensional subspace and following the argument in \cite[Section 3.2]{du2021improved}. The result is that for $\gamma = \eps/C$, there exists a set $B \subset E_1 \times E_2$ with $\mu_1 \times \mu_2 (B) \le \delta^{\gamma}$ such that for every $x \in E_1$ and $\delta$-tube $T$ through $x$, we have
\begin{equation*}
    \mu_2 (T \setminus B|_x) \lesim \delta^{k-1-\eps}.
\end{equation*}
    The set $A$ is again not needed in this case. The main novelty of Theorem \ref{conj:threshold} comes when $r > \delta$.
\end{remark}

We will also need the following two lemmas from \cite{KevinRadialProj} (Lemmas 7.5 and 7.8) governing the physical location of small plates with large mass. 

\begin{lemma}\label{lem:few_large_plates}
    Let $k-1 < s \le k$ and $0 < r \le 1$. There is $N=N(s, k, d)$ such that the following holds: let $\nu$ be an $s$-dimensional measure with constant $C_\nu \ge 1$, and let $\cE_{r,k}$ be the collection of essentially distinct $(r, k)$-plates from the Notations part of Section 1. Let $\cH = \{ H \in \cE_{r,k} : \nu(H) \ge a \}$. Then $|\cH| \lesim (\frac{C_\nu}{a})^N$. (The implicit constant only depends on $k, d$ and is independent of $a, r$.)
\end{lemma}

\begin{lemma}\label{lem:concentrate_in_r0}
    Let $0<r<r_0\lesssim 1$ and $s > k-1$. Let $\cH$ be a collection of $(r,k)$-plates, and let $\mu$ be a compactly supported $s$-dimensional measure with constant $C_\mu$. Then for all $x \in \spt(\mu)$ except a set of $\mu$-measure $\lesssim C_\mu \left( \frac{r}{r_0} \right)^{s-(k-1)} |\cH|^2$, there exists an $(r_0,k)$-plate that contains every $(r,k)$-plate in $\cH$ that passes through $x$.
\end{lemma}

\section{Construction of good measure and Proposition \ref{mainest1}}\label{sec:prop2.1}
\setcounter{equation}0

In this section, we will construct the good measure $\mu_{1,g}^x$ and prove Proposition \ref{mainest1}. We will henceforth treat $\alpha, k, d, \eps$ in the hypothesis of Proposition \ref{mainest1} as fixed constants. To assist in the proof, we will eventually be choosing the following parameters: $\eps_0 (\eps)$, $\kappa(\eps_0)$, $\eta(\kappa, \eps_0)$, $\beta(\kappa, \eta, \eps_0)$. In terms of size, they satisfy
\begin{equation*}
    0 < \beta \ll \eta \ll \kappa \ll \eps_0 \ll \eps.
\end{equation*}
Here, $A \ll B$ means ``$A$ is much smaller than $B$.'' Unwrapping the dependences, we see that $\beta$ ultimately only depends on $\eps$, which is what we want.

\subsection{Smooth Partitions of Unity}\label{sec:smooth partitions}
We follow the first part of \cite[Section 3.1]{du2021improved}. Let $R_0$ be a large power of $2$ that will be determined later, and let $R_j = 2^j R_0$. Construct a partition of unity
\begin{equation*}
    1 = \sum_{j \ge 0} \psi_j,
\end{equation*}
where $\psi_0$ is supported in the ball $|\omega| \le 2R_0$ and each $\psi_j$ for $j \ge 1$ is supported on the annulus $R_{j-1} \le |\omega| \le R_{j+1}$. Importantly, we may choose $\psi_j$ such that $\norm{\cpsi_j}_{L^1} \le C$ for some absolute constant $C$ and all $j \ge 1$.
For example, choose $\psi_j$ to be Littlewood-Paley functions $\chi(x/R_{j}) - \chi(x/R_{j-1})$, where $\chi$ is a smooth bump function that is $1$ on $B(0, 1)$ and $0$ outside $B(0, 2)$.

In $\R^d$, cover the annulus $R_{j-1} \le |\omega| \le R_{j+1}$ by rectangular blocks $\tau$ of dimensions approximately $R_j^{1/2} \times \cdots \times R_j^{1/2} \times R_j$, with the long direction of each block $\tau$ being the radial direction. Choose a smooth ``partition of unity'' with respect to this cover such that
\begin{equation*}
    \psi_j = \sum_{\tau} \psi_{j,\tau} (\omega).
\end{equation*}
The functions $\psi_{j,\tau}$ satisfy the following properties:
\begin{itemize}
    \item $\psi_{j,\tau}$ is supported on $\tau$ and $\norm{\psi_{j,\tau}}_{L^\infty} \le 1$;

    \item $\cpsi_{j,\tau}$ is essentially supported on a $R_j^{-1/2} \times \cdots \times R_j^{-1/2} \times R_j^{-1}$ box $K$ centered at $0$, in the sense that $|\cpsi_{j,\tau} (x)| \le \RapDec(R_j)$ if $\dist (x, K) \gesim R_j^{-1/2+\beta}$ (and the implicit constant in the decay $C_N R_j^{-N}$ is universal only depending on $N$);

    \item $\norm{\cpsi_{j,\tau}}_{L^1} \lesim 1$ (the implicit constant is universal).
\end{itemize}

For each $(j, \tau)$, cover the unit ball in $\R^d$ with tubes $T$ of dimensions approximately $R_j^{-1/2 + \beta} \times\cdots \times R_j^{-1/2+\beta} \times 20$ with the long axis parallel to the long axis of $\tau$. The covering has uniformly bounded overlap, each $T$ intersects at most $C(d)$ other tubes. We denote the collection of all these tubes as $\mathbb{T}_{j, \tau}$. Let $\eta_T$ be a smooth partition of unity subordinate to this covering, so that for each choice of $j$ and $\tau$, $ \sum_{T \in \mathbb{T}_{j, \tau}} \eta_T $ is equal to 1 on the ball of radius 10 and each $\eta_T$ is smooth.

For each $T \in \mathbb{T}_{j, \tau}$, define an operator
\[
M_T f := \eta_T (\psi_{j, \tau} \hat f)^{\vee},
\]
which, morally speaking, maps $f$ to the part of it that has Fourier support in $\tau$ and physical support in $T$.  Define also $M_0 f := (\psi_0 \hat f)^{\vee}$.  We denote $\mathbb{T}_j = \cup_{\tau} \mathbb{T}_{j, \tau}$ and $\mathbb{T} = \cup_{j \ge 1} \mathbb{T}_j$. Hence, for any $L^1$ function $f$ supported on the unit ball, one has the decomposition
\[
f = M_0 f + \sum_{T \in \mathbb{T}} M_T f+\text{RapDec}(R_0)\|f\|_{L^1}.
\]
See \cite[Lemma 3.4]{guth2020falconer} for a justification of the above decomposition. (Even though \cite[Lemma 3.4]{guth2020falconer} is stated in two dimensions, the argument obviously extends to higher dimensions.)

\subsection{Heavy Plates and Good Tubes} \label{sec: good}
In this subsection, we define good tubes. Actually, we will use three categories: good, acceptable, and non-acceptable. The idea is that Theorem \ref{conj:threshold} tells us that $R_j^{-1/2+\beta}$-tubes fall into one of three categories:
\begin{itemize}
    \item For $R_j^{-\eta}$-concentrated $(R_j^{-\kappa},k)$-plates, tubes in them can have large $\mu_2$-mass. Call a tube inside one of these plates non-acceptable.

    \item Many of the acceptable tubes $T$ are good, i.e. $\mu_2 (4T) \lesim R_j^{-\alpha/2 + \eps_0}$.

    \item By Theorem \ref{conj:threshold}, there are not many tubes that are neither non-acceptable nor good.
\end{itemize}

The idea is that to form our good measure $\mu_{1,g}^x$, we keep contributions only from good tubes. By the third bullet, we are allowed to remove tubes that are neither non-acceptable nor good. To remove the non-acceptable tubes, we will instead remove the heavy plates. Next, we formalize this idea.

Let $k$ be the integer such that $k-1<\alpha\leq k$. Let $\cE_{r,k}$ be the cover of $B(0, 1)$ with $(r,k)$-plates as described in the Introduction; every $(r/2, k)$-plate is contained within some element of $\cE_{r,k}$. Let $\cH_j$ be the set of $(R_j^{-\kappa}, k)$-plates in $\cE_{R_j^{-\kappa},k}$ that are $R_j^{-\eta}$-concentrated on $\mu_1 + \mu_2$; then Lemma \ref{lem:few_large_plates} tells us $|\cH_j| \lesim R_j^{N\eta}$. Let $\cH_{\le j} = \cup_{i=1}^j \cH_i$ and $\cH = \cup_{i=1}^\infty \cH_i$. Note that $|\cH_{\le j}| \lesim R_j^{N\eta}$.


Let $\Csep \ge 1$ be a constant such that $\dist(E_1, E_2) \ge \Csep^{-1} \gesim 1$. We will eventually choose a ``master parameter'' $a\in [99\Csep, 100\Csep]$. For $H\in \cH$, we will use $aH$ as proxies for $H$ when defining acceptable tubes and the good measure. 

We briefly attempt to motivate the construction. The role of $\Csep$ is to make sure that tubes intersecting $H \cap E_1$ and $H \cap E_2$ actually lie inside $a H$. The role of $a$ is to introduce a probabilistic wiggle to make a key technical condition hold (see the control of $\text{Bad}_j^2$ in Lemma \ref{lem:bad_estimate}).

Now, we fix a choice of $a$ and define the following. We say an $R_j^{-1/2+\beta}$-tube $T\in\T_j$ is \emph{non-acceptable} if there exists some $H\in\cH_{\leq \max(j, j_*)}$ such that $2T$ is contained in $a H$, where $j_* = \log_2 R_0$ (the motivation for introducing $j_*$ will be explained in the proof of Lemma \ref{lem:refine E_2}). Otherwise, we say it is \emph{acceptable}. Define a \emph{good} $R_j^{-1/2+\beta}$-tube $T$ to be an acceptable tube with $\mu_2 (4T) \le R_j^{-\alpha/2+\eps_0}$. And define the good part of $\mu_1$ with respect to $\mu_2$ and $x \in E_2$ by
$$
\mu_{1,g}^x := M_0 (\mu_1|_{G_0 (x)}) + \sum_{T\in\T,\, T \text{ good}} M_T \mu_1.
$$
The only dependence on $x$ comes in the $M_0 (\mu_1|_{G_0 (x)})$, which is crucial when we try to prove Proposition \ref{mainest2} later. We will define $G_0 (x)$ in the next subsection, roughly speaking, $G_0 (x)$ is obtained by removing heavy plates through $x$ at several scales. 

As constructed, we may not get a good bound of the form $\norm{d^x_* (\mu_{1,g}^x) - d^x_* (\mu_1)} \le R_0^{-\eps}$. This is because $\mu_{1,g}^x$ doesn't include contributions from non-acceptable tubes while $\mu_1$ does. Instead, we need to work with a measure $\mu_1^x$ depending on $x$ that removes the contributions of the non-acceptable tubes through $x$. In fact, since non-acceptable tubes are contained in heavy plates, we should define $\mu_1^x$ by removing these heavy plates ``at different scales $R_j$'' and make sure that summing over different scales still leads to good behavior (see Lemma \ref{lem:Ginfty}). We make things rigorous in the next subsection. 

\subsection{Construction of $G(x)$ and $\mu_1^x$}\label{subsec:construct G(x)}
Recall that $\dist(E_1, E_2) \ge \Csep^{-1} \gesim 1$. Thus, for $x\in E_2$, we have $E_1 
\subset B(x, \Csep^{-1})^c$, a fact that underlies the rest of the paper.

Recall that $a \in [99\Csep, 100\Csep]$ is the ``master parameter'' to be chosen later. For $x \in E_2$ and $H \in \cH_j (x)$, let $F(x, aH)$ be given by
$$
F(x, aH):=\{y\in B(x, \Csep^{-1})^c \cap B(0,10) : l(x,y) \cap \surf(aH) =\emptyset\}\,,
$$
where $l(x,y)$ is the line through $x$ and $y$. It is true that $F(x, aH) \subset aH$; this will be proved in Lemma \ref{lem:acceptable non-acceptable}. Define $j_*:=\log_2 R_0$ such that $R_{j_*} = R_0^2$. For $j \ge 0$, let
\begin{equation}\label{eqn:gjx defn}
    G_j (x) = \left[ B(x, \Csep^{-1})^c \cap B(0,10) \right]\setminus \cup_{H \in \cH_{\le \max(j, j_*)}(x)} F(x, aH) \,.
\end{equation}
Finally, we define $G(x) = G_0 (x) \cup B(x, \Csep^{-1})$ and
\begin{equation}\label{eqn:mu1x defn}
    \mu_1^x := \sum_{j \ge 0} \mu_1|_{G_j (x)} * \cpsi_j.
\end{equation}
It will be proved that $G(x)$ satisfies the condition of Proposition \ref{mainest1} in the next subsection, and this is the only reason why we include $B(x, \Csep^{-1})$ in $G(x)$.

We now list some good properties of these definitions that will be critical later. First, the construction of $G_j (x)$ ensures that $G_j (x) \supset G_{j+1} (x)$ and $$G_j (x) \setminus G_{j+1} (x) \subset \cup_{H \in \cH_{j+1} (x)} F(x, aH)\,,$$ with $G_j (x) = G_{j+1} (x)$ for $j < j_*$. Next, $G_j (x)$ keeps the contributions from the acceptable tubes through $x$ while discarding the non-acceptable tubes through $x$.

\begin{lemma}\label{lem:acceptable non-acceptable} Let $T\in \T_j$ and $x\in 2T \cap E_2$.

    (a) If $2T \cap B(x, \Csep^{-1})^c \subset G_j (x)$, then $T$ is acceptable;

    (b) If $2T \cap B(x, \Csep^{-1})^c \subset F(x, aH)$ for some $H \in \cH_{\le \max(j,j_*)}(x)$, then $T$ is non-acceptable.

\end{lemma}

\begin{proof}
Let $T\in \T_j$ and $x\in 2T$.

(a) If $2T \cap B(x, \Csep^{-1})^c \subset G_j (x)$, we show $T$ is acceptable, i.e. $2T \not\subset aH$ for any $H \in \cH_{\le \max(j, j_*)}$. Fix $H \in \cH_{\le \max(j, j_*)}$. If $x \notin aH$, then $2T \not\subset aH$. If $x \in aH$, then $H \in \cH_{\le \max(j, j_*)} (x)$. Choose $y \in 2T \cap B(x, \Csep^{-1})^c$ such that the line $l(x,y)$ is parallel to the central line of $T$. By assumption, we know that $y \in G_j (x)$, so $l(x,y)$ intersects $\surf(aH)$. In particular, $2T$ intersects $\surf(aH)$, and thus $2T \not\subset aH$.

(b) If $2T \cap B(x, \Csep^{-1})^c \subset F(x, aH)$ for some $H \in \cH_{\ell} (x)$ with $\ell \le \max(j, j_*)$, we show $T$ is non-acceptable. More precisely, we can show that $2T$ is contained in $aH$.

As promised before, we will prove $F(x, aH) \subset aH$. Take $y \in F(x, aH)$, and suppose $y \notin aH$. Then $d(y, P_H) \geq aR_\ell^{-\kappa}$ and $d(x, P_H) < aR_\ell^{-\kappa}$. By continuity, we can find $z$ on the line segment between $x$ and $y$ such that $d(z, P_H) = aR_\ell^{-\kappa}$. By convexity of $B(0, 10)$, we have $z \in B(0, 10)$. Thus, we have $z \in \surf(aH)$, which contradicts $y \in F(x, aH)$, and so in fact $y \in aH$. This proves $F(x, aH) \subset aH$.

Now observe the following geometric fact: $2T$ is a tube through $x \in B(0, 1)$ whose ends lie on $B(0, 10)$, and so $2T \cap B(x, \Csep^{-1})^c$ contains both ends of $2T$. Also, by our initial assumption, we get
\begin{equation*}
    2T \cap B(x, \Csep^{-1})^c \subset F(x, aH) \subset aH.
\end{equation*}
Therefore, since $aH$ is convex, we get $2T\subset aH$.
\end{proof}

To estimate $\norm{d_*^x(\mu_1|_{G(x)}) - d_*^x(\mu_{1,g}^x)}_{L^1}$, it suffices to estimate $\norm{\mu_1|_{G(x)} - \mu_1^x}_{L^1}$ and $\norm{d_*^x(\mu_1^x) - d_*^x(\mu_{1,g}^x)}_{L^1}$. This will be the content of the next two subsections.

\subsection{Pruning $E_2$ and estimating $\norm{\mu_1|_{G(x)} - \mu_1^x}_{L^1}$} \label{sec:pruning}
If all the $G_j (x)$'s were the same, then $\mu_1|_{G(x)} = \mu_1|_{G_0 (x)} = \mu_1^x$. This isn't quite true, but in general, we will still get a good bound on $\norm{\mu_1|_{G(x)} - \mu_1^x}_{L^1}$ if $\mu_1 (G_j (x) \setminus G_{j+1} (x))$ is small for all $j$. This weaker assumption is in fact true for ``most'' $x \in E_2$. To see this, we use a variant of Proposition 7.3 in \cite{KevinRadialProj} (see also Proposition B.1 of \cite{shmerkin2022non}).
\begin{lemma}\label{lem:refine E_2}
    For every $\kappa > 0$, there exists $\eta_0 (\alpha, k, \kappa) > 0$ such that for all $\eta \in (0, \eta_0]$ and $R_0$ sufficiently large in terms of $\alpha, k, \kappa, \Csep, \eta$, the following holds. Then there exists $E_2'' \subset E_2$ with $\mu_2 (E_2 \setminus E_2'') \le R_0^{-\eta}$ such that the following assertions hold for $x \in E_2''$.
    \begin{enumerate}[(a)]
        \item The $100\Csep$-scalings of elements in $\cH_j (x)$ are contained within some $(R_j^{-\kappa/2}/3, k)$-plate through $x$;

        \item The $100\Csep$-scalings of elements in $\cH_{\le j_*} (x)$ are contained within some $(R_0^{-\kappa/2}, k)$-plate through $x$;

        \item $\mu_1 (G_j (x) \setminus G_{j+1} (x)) \lesim R_j^{-\eta/2}$, for any $j\geq 0$.
    \end{enumerate}
\end{lemma}

\begin{proof}
    First pick $j \ge 0$. By Lemma \ref{lem:concentrate_in_r0} and $300\Csep < R_0^{\kappa/2}$ for sufficiently large $R_0$, we can find sets $F_j$ with $\mu_2 (F_j) \lesim R_j^{-\kappa(\alpha-k+1)/2} |\cH_j|^2$ such that for $x \in E_2 \setminus F_j$, the $100\Csep$-scalings of elements in $\cH_j (x)$ are contained within some $(R_j^{-\kappa/2}/3, k)$-plate through $x$. Thus for $x \in E_2 \setminus F_j$, assertion (a) is true. 

    By the same Lemma, we can find a set $F_0$ with $$\mu_2 (F_0) \lesim R_0^{-\kappa(\alpha-k+1)/2} |\cH_{\le j_*}|^2$$ such that for $x \in E_2 \setminus F_0$, the $100\Csep$-scalings of elements in $\cH_{\le j_*} (x)$ are contained within some $(R_0^{-\kappa/2}, k)$-plate through $x$. (Some of the plates in $\cH_{\le j_*}$ are too small, but we simply thicken them to $(R_0^{-\kappa}, k)$-plates.) Thus for $x \in E_2 \setminus F_0$, assertion (b) is true.
    
    Let $E_2'' = E_2 \setminus (F_0 \cup \bigcup_{j=j_*}^\infty F_j)$; then (using $R_{j_*} = R_0^2$ and Lemma \ref{lem:few_large_plates}):
    \begin{multline*}
        \mu_2 (E_2 \setminus E_2'') \lesim R_0^{-\kappa(\alpha-k+1)/2} |\cH_{\le j_*}|^2 + \sum_{j \ge j_*} R_j^{-\kappa(\alpha-k+1)/2} |\cH_j|^2 \\
        \lesim R_0^{-\kappa(\alpha-k+1)/2+4N\eta} + \sum_{j \ge j_*} R_j^{-\kappa(\alpha-k+1)/2 + 2N\eta} \lesim \sum_{j \ge 0} R_j^{-2\eta}\lesim_\eta R_0^{-2\eta},
    \end{multline*}
    if $\eta \le \eta_0$ is chosen sufficiently small in terms of $\kappa$. Hence, if $R_0$ is chosen sufficiently large in terms of $\eta$, then $R_0^\eta$ dominates the implicit constant and we get $\mu_2 (E_2 \setminus E_2'') \le R_0^{-\eta}$. Thus, $E_2'' \subset E_2$ satisfies the desired bound and assertions (a) and (b) are proved.
    
    Let $x \in E_2''$. For assertion (c), we observe that if $R_j < R_0^2$, then $G_j (x) = G_{j+1} (x)$. (This is the reason why the parameter $j_*$ was introduced.) Thus, assume $R_j \ge R_0^2$. Then $$G_j (x) \setminus G_{j+1} (x) \subset \bigcup_{H \in \cH_{j+1} (x)} F(x, aH)\,,$$ which is contained within some $(R_{j+1}^{-\kappa/2}/3, k)$-plate $V$ through $x$ using assertion (a). Let $m$ be such that $R_{j+1} \in [R_m^2/2, 2R_m^2]$. Since $R_{j+1}^{-\kappa/2}/3 \le R_m^{-\kappa}/2$, we know that $V$ is contained in some $(R_m^{-\kappa}, k)$-plate $H \in \cE_{R_m^{-\kappa}, k}$. Note that, since $a \ge 99\Csep$, we have $H \cap B(x, \Csep^{-1})^c \subset F(x, aH)$. Therefore, if $H \in \cH_m$, then $V \cap G_m (x) = \emptyset$. Otherwise, we have $\mu_1 (V) \le R_m^{-\eta}$ by definition of $\cH_m$. In either case, we have
    $$\mu_1(G_j (x) \setminus G_{j+1} (x) )\le \mu_1 (V \cap G_m (x)) \lesssim R_m^{-\eta} \lesssim R_j^{-\eta/2}\,,$$ as desired.
\end{proof}

Lemma \ref{lem:refine E_2} has two ramifications. First, Lemma \ref{lem:refine E_2}(b) tells us that $B(0, 10) \setminus G(x)$ is contained within some $(R_0^{-\kappa/2}, k)$-plate. Second, Lemma \ref{lem:refine E_2}(c) tells us that $\mu_1 (G_j (x) \setminus G_{j+1} (x))$ is small for all $x \in E_2''$, so we can estimate $\norm{\mu_1^x - \mu_1|_{G(x)}}_{L^1}$. We record these two observations and provide detailed proofs in the following lemma.



\begin{lemma}\label{lem:Ginfty}
    Let $E_2''$ be the subset given in Lemma \ref{lem:refine E_2} and $x \in E_2''$. Then $B(0, 10) \setminus G(x)$ is contained within some $(R_0^{-\kappa/2}, k)$-plate and
    \begin{equation*}
        \norm{\mu_1|_{G(x)} - \mu_1^x}_{L^1} \lesim R_0^{-\eta/2}.
    \end{equation*}
\end{lemma}


\begin{proof}
    For the first assertion, we use the definition of $G(x), G_0 (x)$ in \eqref{eqn:gjx defn} and the fact $F(x, aH) \subset aH$ proved in Lemma \ref{lem:acceptable non-acceptable} to write
    \begin{equation*}
        B(0,10) \setminus G(x) = \left[B(x, \Csep^{-1})^c \cap B(0,10)\right]\setminus G_0 (x) \subset \bigcup_{H \in \cH_{\leq j_*} (x)} aH.
    \end{equation*}
    By Lemma \ref{lem:refine E_2}(b), the rightmost expression is contained within some $(R_0^{-\kappa/2}, k)$-plate.

    For the second assertion, let $G_\infty (x) = \bigcap_{j=0}^\infty G_j (x)$. First, by Lemma \ref{lem:refine E_2}(c), we establish that for all $j \ge 0$,
    \begin{equation}\label{eqn:G_0 close to G_infty}
        \mu_1 (G_j (x) \setminus G_\infty (x)) = \sum_{i=j}^\infty \mu_1 (G_i (x) \setminus G_{i+1} (x)) \lesim \sum_{i=j}^\infty R_i^{-\eta/2} \lesim R_j^{-\eta/2}.
    \end{equation}
    Also, $\mu_1$ is supported on $B(x, \Csep^{-1})^c$, so $\mu_1|_{G(x)}$ and $\mu_1|_{G_0 (x)}$ are the same measure. Thus, by \eqref{eqn:G_0 close to G_infty}, we get $\norm{\mu_1|_{G(x)} - \mu_1|_{G_\infty (x)}}_{L^1} \lesim R_0^{-\eta/2}$, so it suffices to show that $\norm{\mu_1|_{G_\infty (x)} - \mu_1^x}_{L^1} \lesim R_0^{-\eta/2}$. Indeed, using $\sum_{j\ge 0} \psi_j = 1$ and the definition \eqref{eqn:mu1x defn} of $\mu_1^x$, we have
    \begin{equation*}
        \mu_1^x - \mu_1|_{G_\infty (x)} = \sum_{j \ge 0} (\mu_1|_{G(x)} - \mu_1|_{G_\infty (x)}) *  \cpsi_j = \sum_{j \ge 0} \mu_1|_{G_j (x) \setminus G_\infty (x)} * \cpsi_j,
    \end{equation*}
    and so by $\norm{\cpsi_j}_{L^1} \lesim 1$, Young's convolution inequality, and \eqref{eqn:G_0 close to G_infty}, we have
    \begin{equation*}
        \norm{\mu_1^x - \mu_1|_{G_\infty (x)}}_{L^1} \lesim \sum_{j\ge 0} \mu_1 (G_j (x) \setminus G_\infty (x)) \lesim \sum_{j \ge 0} R_j^{-\eta/2} \lesim R_0^{-\eta/2}.
    \end{equation*}
\end{proof}

\begin{remark}
    We can make the last step more efficient using Abel summation. Let $\Check{P}_k = \sum_{j=1}^k \cpsi_j$. Note that if we chose $\psi_j$ to be Littlewood-Paley functions $\chi(x/R_j) - \chi(x/R_{j-1})$ as in Section \ref{sec:smooth partitions}, then $\norm{\Check{P}_k}_{L^1} \le C$ for some universal constant $C$. Now, Abel summation gives
    \begin{equation*}
        \mu_1^x - \mu_1|_{G_\infty (x)} = \sum_{j=0}^\infty \mu_1|_{G_j (x) \setminus G_{j+1} (x)} * \Check{P}_j,
    \end{equation*}
    and thus we get a slightly sharper bound $\norm{\mu_1^x - \mu_1|_{G_\infty (x)}}_{L^1} \lesim \mu(G_0 (x) \setminus G_\infty (x))$. This can be useful in some potential applications where $\mu(G_0 (x) \setminus G_\infty (x))$ is controlled but not $\mu(G_j (x) \setminus G_\infty (x))$ as $j \to \infty$.
\end{remark}

\subsection{Tube Geometry and bound of $\norm{d_*^x (\mu_1^x) - d_*^x (\mu_{1,g}^x)}_{L^1}$}\label{sec:tube_geometry}

The goal of this subsection is to establish a bound on $\norm{d_*^x (\mu_1^x) - d_*^x (\mu_{1,g}^x)}_{L^1}$ in terms of the geometry of the tubes (Lemma \ref{lem:compare dstar}). Showing such a bound will allow us to prove Proposition \ref{mainest1}, thanks to Lemma \ref{lem:Ginfty}.

Recall that
\begin{equation*}
    \mu_1^x = \sum_{j \ge 0} \mu_1|_{G_j (x)} * \cpsi_j.
\end{equation*}

We first give a good approximation to $\mu_1^x$. The following lemma can be proved the same way as Lemma 3.4 in \cite{guth2020falconer} and we omit the details. It shows that problems about the mysterious $\mu_1^x$ actually reduce to problems about restricted measures of $\mu_1$. We remark that the next few lemmas will be proved for all $x \in E_2$, even though we only need these results for $x \in E_2''$ to prove Proposition \ref{mainest1}. This is a minor difference: restricting to $x \in E_2''$ doesn't make the lemmas easier to prove. 
\begin{lemma}\label{lem:tube_decomposition}
Let $x \in E_2$. Then
\begin{equation*}
    \norm{\mu_1^x -M_0(\mu_1|_{G_0 (x)})-\sum_{j=1}^\infty \sum_{T \in \T_j} M_T (\mu_1|_{G_j (x)})}_{L^1} \le \RapDec(R_0).
\end{equation*}
\end{lemma}

We make the following observations relating the geometry of tubes to analytic estimates.

\begin{lemma}\label{lem:types of tube}
Let $T\in \T_j$ be an $R_j^{-1/2+\beta}$-tube and $x\in E_2$.
\begin{enumerate}[(a)]
    \item If $2T$ doesn't pass through $x$, then $\norm{d_*^x (M_T \mu_1)}_{L^1} \le \RapDec(R_j)$ and $\norm{d_*^x (M_T (\mu_1|_{G_j (x)}))}_{L^1} \le \RapDec(R_j)$.
    
    \item If $2T \cap B(x, \Csep^{-1})^c \subset G_j (x)$, then $\norm{M_T \mu_1 - M_T (\mu_1|_{G_j (x)})}_{L^1} \le \RapDec(R_j)$.

    \item If $2T \cap B(x, \Csep^{-1})^c \subset F(x, aH)$ for some $H \in \cH_{\leq \max(j,j_*)}(x)$, then $\norm{M_T (\mu_1|_{G_j (x)})}_{L^1} \le \RapDec(R_j)$.

    \item $\norm{M_T (\mu_1|_{G_j (x)})}_{L^1}$ and $\norm{M_T \mu_1}_{L^1}$ are both $\lesim \mu_1 (2T) + \RapDec(R_j)$.
\end{enumerate}
\end{lemma}

\begin{proof}
    (a) This is Lemma 3.1 of \cite{guth2020falconer} applied to $\mu_1$ and $\mu_1|_{G_j (x)}$; note that $\norm{\mu_1|_{G_j (x)}}_{L^1} \le \norm{\mu_1}_{L^1} \le 1$.

    (b) By assumption, $\mu_1 - \mu_1|_{G_j (x)}$ is supported outside $2T \cap B(x, \Csep^{-1})^c \supset 2T \cap E_1$, so we can apply Lemma 3.2 of \cite{guth2020falconer}.

    (c) By assumption, $\mu_1|_{G_j (x)}$ is supported outside $2T \cap E_1$, so we can apply Lemma 3.2 of \cite{guth2020falconer}.

    (d) This is a direct consequence of Lemma 3.2 of \cite{guth2020falconer}.
\end{proof}

Using Lemma \ref{lem:types of tube}, we are able to compare $\mu_1^x$ with $\mu_{1,g}^x$. We first need some definition. For $x \in E_2$, $j\geq 0$, define $\Bad_j (x)$ to be the union of $2T$, where $T\in\T_j$ is an $R_j^{-1/2+\beta}$-tube such that $2T$ passes through $x$ and either (1) $2T \cap B(x, \Csep^{-1})^c$ is not contained in $G_j (x)$ or any $F(x, aH)$ for $H \in \cH_{\le \max(j, j_*)}(x)$; or (2) $2T \cap B(x, \Csep^{-1})^c \subset G_j (x)$ and $\mu_2 (4T) > R_j^{-\alpha/2+\eps_0}$. By Lemma \ref{lem:acceptable non-acceptable}, (2) is morally the union of the acceptable but not good tubes through $x$, while (1) is the union of the tubes through $x$ that are ``borderline'' between acceptable and non-acceptable.

One may compare the following lemma with Lemma 3.1 of \cite{du2021improved}.

\begin{lemma}\label{lem:compare dstar}
Let $x\in E_2$ and $\text{Bad}_j(x)$ be defined as above, $\forall j\geq 0$. Then,
\begin{equation*}
    \norm{d_*^x (\mu_1^x) - d_*^x (\mu_{1,g}^x)}_{L^1} \lesim \sum_{j \ge 1} R_j^{100\beta d} \mu_1 (\Bad_j (x)) + \RapDec(R_0).
\end{equation*}    
\end{lemma}

\begin{proof}
We apply Lemma \ref{lem:tube_decomposition} and the definition of $\mu_{1,g}^x$. Note that the $M_0 (\mu_1|_{G_0(x)})$ terms cancel in the resulting expression for $\mu_1^x - \mu_{1,g}^x$ (this is the reason why $\mu_{1,g}^x$ needs to depend on $x$). Let $g(T) = 1$ if $T$ is good and $0$ otherwise. Thus, it suffices to prove for each $j \ge 1$,
\begin{align}
    &\norm{\sum_{T \in \T_j} [d^x_* (M_T (\mu_1|_{G_j (x)})) - d^x_* (M_T \mu_1) g(T)]}_{L^1} \notag \\
    \lesim & R_j^{100\beta d} \mu_1 (\Bad_j (x)) + \RapDec(R_j). \label{eqn:compare dstar}
\end{align}
Let $\Tjbad$ be the set of tubes $T \in \T_j$ such that $2T$ passes through $x$ and either condition (1) or (2) in the definition of $\text{Bad}_j(x)$ holds. We claim that if $T \notin \Tjbad$, then
\begin{equation}\label{eqn:bound1}
    \norm{d^x_* (M_T (\mu_1|_{G_j (x)})) - d^x_* (M_T \mu_1) g(T)}_{L^1} \le \RapDec(R_j),
\end{equation}
while if $T \in \Tjbad$, then
\begin{equation}\label{eqn:bound2}
\norm{d^x_* (M_T (\mu_1|_{G_j (x)})) - d^x_* (M_T \mu_1) g(T)}_{L^1} \lesssim \mu_1 (2T) + \RapDec(R_j).
\end{equation}
We show how the claim implies \eqref{eqn:compare dstar}. For any $y \in E_1$, $d(x, y) \gesim 1$ and so there are $\lesim R_j^{100\beta d}$ many $R_j^{-1/2+\beta}$-tubes in $\T_j$ passing through both $x$ and $y$. Thus,
\begin{equation}\label{eqn:bound3}
\sum_{T \in \Tjbad} \mu_1 (2T) \lesim R_j^{100\beta d} \mu_1 (\Bad_j (x)).
\end{equation}
Combining \eqref{eqn:bound1}, \eqref{eqn:bound2}, \eqref{eqn:bound3} proves \eqref{eqn:compare dstar}.

Now we prove the claim. Suppose $T \notin \Tjbad$. By working through the definition of $\Tjbad$, we have three possibilities of $T$: either
\begin{enumerate}[(i)]
    \item $x \notin 2T$;
    \item $x \in 2T$, $2T \cap B(x, \Csep^{-1})^c \subset G_j (x)$ and $\mu_2 (4T) \le R_j^{-\alpha/2+\eps_0}$. Then $T$ is acceptable by Lemma \ref{lem:acceptable non-acceptable}(a), so it is good since $\mu_2 (4T) \le R_j^{-\alpha/2+\eps_0}$.
    \item $x \in 2T$, $2T \cap B(x, \Csep^{-1})^c \subset F(x, aH)$ for some $H \in \cH_{\le \max(j, j_*)}(x)$. Then $T$ is non-acceptable by Lemma \ref{lem:acceptable non-acceptable}(b).
\end{enumerate}

In case (i), we get \eqref{eqn:bound1} by Lemma \ref{lem:types of tube}(a), regardless of $g(T) = 0$ or $1$. In case (ii), we use Lemma \ref{lem:types of tube}(b) and $g(T) = 1$, and in case (iii), we use Lemma \ref{lem:types of tube}(c) and $g(T) = 0$. Thus, if $T \notin \Tjbad$, then \eqref{eqn:bound1} holds.

Now suppose $T \in \Tjbad$. Then we get \eqref{eqn:bound2} by applying Lemma \ref{lem:types of tube}(d), regardless of whether $g(T) = 0$ or $1$. This proves the claim.
\end{proof}

The crucial estimate about $\Bad_j (x)$ is the following lemma, which will be proved in Section \ref{sec:bad measure}.

\begin{lemma}\label{lem:bad_estimate}
    For every $\eps_0 > 0$, there exist $\eta_0 (\eps_0), \kappa_0 (\eps_0) > 0$ such that for any $\eta \in (0, \eta_0], \kappa \in (0, \kappa_0]$ and sufficiently small $\beta$ depending on $\eps_0, \eta, \kappa$, the following holds. In the construction of $\mu_{1,g}^x$ in Section \ref{sec: good}, we can choose some $a \in [99\Csep, 100\Csep]$ such that for any $j\geq 1$, if we define $$\Bad_j := \{ (x, y) \in E_2\times E_1 : y \in \Bad_j (x)\}\,,$$ then $\mu_2 \times \mu_1 (\Bad_j) \lesssim R_j^{-200\beta d}.
    $
\end{lemma}

Using this, we complete the proof of Proposition \ref{mainest1}.

\begin{proof}[Proof of Proposition \ref{mainest1}]
Our goal is to find a large subset $E_2' \subset E_2$ with $\mu_2(E_2') \ge 1 - R_0^{-\beta}$ such that for each $x \in E_2'$, we have that $B^d(0,10) \setminus G(x)$ is contained within some $(R_0^{-\beta}, k)$-plate and
\begin{equation} \label{eqn:mu1good1'}
        \norm{d_*^x (\mu_1|_{G(x)}) - d_*^x (\mu_{1,g}^x)}_{L^1} \le R_0^{-\beta}.
\end{equation}

First, let us determine the auxiliary parameters $\eps_0, \kappa, \eta, \beta$. (Recall that $\eps, \alpha, k, d$ are fixed constants.) We defer the choice of $\eps_0 (\eps)$ to the next section, see Lemma \ref{lem:big_r}. Then, choose $\kappa = \kappa_0 (\eps_0)$ in Lemma \ref{lem:bad_estimate}. Next, choose $\eta = \eta(\kappa, \eps_0)$ to be the smaller of the two $\eta_0$'s in Lemma \ref{lem:refine E_2} and Lemma \ref{lem:bad_estimate}. Finally, choose $\beta$ to the smaller of the $\beta(\eps_0, \eta, \kappa)$ in Lemma \ref{lem:bad_estimate} and $\min(\frac{\eta}{3}, \frac{\kappa}{2})$.

Now, we shall construct $E_2'$ by taking the set $E_2''$ from Lemma \ref{lem:refine E_2} and removing some ``bad parts'' given by Lemma \ref{lem:bad_estimate}. Fix a choice of $a$ in the construction of $\mu_{1,g}^x$ in Section \ref{sec: good} such that the conclusion in Lemma \ref{lem:bad_estimate} holds for any $j\geq 1$. By Lemma \ref{lem:bad_estimate}, for each $j\geq 1$, we can find a set $F_j \subset E_2$ with $\mu_2 (F_j) \le R_j^{-50 \beta d}$ such that $\mu_1 (\Bad_j (x)) \lesim R_j^{-150 \beta d}$ for $x\in E_2\setminus F_j$. Finally, define $E_2' := E_2'' \setminus \bigcup_{j \ge 1} F_j$. We now verify that $E_2'$ satisfies the desired conditions.

First, observe that $\mu_2 (E_2') \ge \mu(E_2) - \mu(E_2\setminus E_2'') - \sum_{j \ge 1} R_j^{-50\beta d} > 1-R_0^{-\beta}$ if $R_0$ is sufficiently large.

Next, fix $x \in E_2'$. Since $x \in E_2''$ and $\beta < \frac{\kappa}{2}$, we get from the first part of Lemma \ref{lem:Ginfty} that $B(0, 10) \setminus G(x)$ is contained in some $(R_0^{-\beta}, k)$-plate. Now by the second part of Lemma \ref{lem:Ginfty}, since $\beta \le \frac{\eta}{3}$, we have that
    \begin{equation} \label{eqn: compare to mu1x}
        \norm{d_*^x (\mu_1|_{G(x)}) - d^x_* (\mu_1^x)}_{L^1} < \frac 12 R_0^{-\beta}.
    \end{equation}

    
    
    For each $x \in E_2'$, Lemma \ref{lem:compare dstar} tells us (for some constant $C$ depending only on our parameters):
    \begin{equation}\label{eqn: compare to mu1good}
        \norm{d_*^x (\mu_1^x) - d_*^x (\mu_{1,g}^x)}_{L^1} \le C \cdot \sum_{j \ge 1} R_j^{-50 \beta d} + \RapDec(R_0) < \frac 12 R_0^{-\beta}
    \end{equation}
    if $R_0$ is sufficiently large. Combining \eqref{eqn: compare to mu1x} and \eqref{eqn: compare to mu1good} via triangle inequality proves the desired equation \eqref{eqn:mu1good1'}.
\end{proof}

\subsection{Control of bad part}\label{sec:bad measure}
The goal of this subsection is to prove Lemma \ref{lem:bad_estimate}. To do so, we will use the new radial projection estimate, Theorem \ref{conj:threshold}.

\begin{proof}[Proof of Lemma \ref{lem:bad_estimate}] Let $\Bad_j = \{ (x, y) \in E_2\times E_1 : y \in \Bad_j (x) \}$. 
By definition of $\Bad_j (x)$ above Lemma \ref{lem:compare dstar}, we have $\Bad_j \subset \Bad_j^1 \cup \Bad_j^2$. Here $\Bad_j^1$ is the set of pairs $(x, y)\in E_2\times E_1$ such that $x, y$ lie in $2T$ for some $R_j^{-1/2+\beta}$-tube $T \in \T_j$ with $2T \cap B(x, \Csep^{-1})^c \subset G_j (x)$ and $\mu_2 (4T) \ge R_j^{-\alpha/2+\eps_0}$. And $\Bad_j^2$ is the set of pairs $(x, y)\in E_2\times E_1$ satisfying that $x, y$ lie in $2T$ for some $R_j^{-1/2+\beta}$-tube $T \in \T_j$ such that $2T \cap B(x, \Csep^{-1})^c$ is not contained in $G_j (x)$ or any $F(x, aH)$ for $H \in \cH_{\le \max(j, j_*)}(x)$.

Our bound of $\mu_2 \times \mu_1 (\Bad^1_j)$ will not depend on the choice of $a$ in the construction of $\mu_{1,g}^x$ in Section \ref{sec: good}, while $\mu_2 \times \mu_1 (\Bad^2_j)$ will. For a given $H \in \cH_\ell$ with $\ell \le \max(j, j_*)$, let $\Bad^2_j (H)$ be the set of pairs $(x, y) \in E_2 \times E_1$ satisfying that $H \in \cH_{\le \max(j, j_*)} (x)$ and $x, y$ lie in $2T$ for some $R_j^{-1/2+\beta}$-tube $T \in \T_j$ such that $2T \cap B(x, \Csep^{-1})^c$ is not contained in $\R^d \setminus F(x, aH)$ or $F(x, aH)$. We'll prove
$$
\mu_2 \times \mu_1 (\Bad^1_j) \lesssim R_j^{-200\beta d} 
$$
and that there exists $a$ such that
\begin{equation}\label{eqn:bad2 estimate}
    \sum_{H \in \cH_{\le \max(j, j_*)}} \mu_2 \times \mu_1 (\Bad^2_j(H)) \lesssim R_j^{-1/8}\,,
\end{equation}
for any $j\geq 0$. Since $\Bad_j^2 \subset \bigcup_{H \in \cH_{\le \max(j, j_*)}} \Bad_j^2 (H)$, these two bounds will prove Lemma \ref{lem:bad_estimate}.

\emph{\textbf{Upper bound of $\mu_2 \times \mu_1 (\Bad^1_j)$.}}  This is a consequence of Theorem \ref{conj:threshold}. By applying Theorem \ref{conj:threshold} with parameters $(\delta, r, \delta^\eta, \eps) = (R_j^{-1/2+\beta}, R_j^{-\kappa}/(200\Csep), R_j^{-\eta}, \eps_0/4)$, we can find $\gamma > 0$ such that the following is true. There exists a set $B\subset E_2\times E_1$ with $\mu_2 \times \mu_1 (B) \le R_j^{-\gamma}$ such that for each $y \in E_1$ and $R_j^{-1/2+\beta}$-tube $T$ with $2T$ containing $y$, we have (assuming $\kappa$ and $\beta$ are chosen small enough in terms of $\eps_0$):
    \begin{equation}\label{eq-goodthred}
        \mu_2 (2T \setminus (A|_y \cup B|_y)) \lesim_{\Csep} \frac{R_j^{(-1/2+\beta)\alpha}}{R_j^{-\kappa (\alpha-k+1)}} \cdot R_j^{(-1/2+\beta)(-\eps_0/4)} \le R_j^{-\alpha/2+\eps_0/2}\,,
    \end{equation}
where $A$ is the set of pairs $(x, y) \in E_2 \times E_1$ satisfying that $x$ and $y$ lie in some $R_j^{-\eta}$-concentrated $(R_j^{-\kappa}/(200\Csep),k)$-plate on $\mu_1 + \mu_2$. Since decreasing the values of $\beta, \gamma$ makes the previous statement weaker, we may assume $200 \beta d = \gamma$. 

Now, observe that since $d(y, E_2) \ge \Csep^{-1} \gesim 1$ for $y\in E_1$ and $\mu_2$ is a probability measure, there are at most $\lesim R_j^{\alpha/2-\eps_0+O(\beta)}$ many tubes $T \in \T_j$ with $2T$ containing $y$ satisfying $\mu_2 (4T) \ge R_j^{-\alpha/2+\eps_0}$. Moreover, we claim the following.

\textbf{Claim 1.}. Let $y\in E_1$. Suppose there exist $x\in E_2$ and $T\in\T_j$ such that $x, y$ lie in $2T$ with $2T \cap B(x, \Csep^{-1})^c \subset G_j (x)$. Then $2T\cap A|_y =\emptyset$. 

 Assuming Claim 1, by \eqref{eq-goodthred} we get
$$
\mu_2 (\Bad^1_j|_y \setminus B|_y) \lesssim R_j^{\alpha/2-\eps_0+O(\beta)} \cdot R_j^{-\alpha/2+\eps_0/2} \leq R_j^{-200\beta d}\,.
$$
Thus, $\mu_2 \times \mu_1 (\Bad^1_j \setminus B) \lesim R_j^{-200\beta d}$, and so $\mu_2 \times \mu_1 (\Bad^1_j) \lesim R_j^{-200\beta d}$.

It remains to prove Claim 1. Let $y\in E_1$, $x\in E_2$, and $T\in\T_j$ be such that $x, y$ lie in $2T$ with $2T \cap B(x, \Csep^{-1})^c \subset G_j (x)$. Suppose $2T\cap A|_y \neq \emptyset$. Pick a point $x' \in 2T \cap A|_y$; by definition, we have $x' \in E_2\cap 2T$ and there exists a $R_j^{-\eta}$-concentrated $(R_j^{-\kappa}/(200\Csep),k)$-plate $H'$ such that $x'$ and $y$ lie in $H'$. We also know $x', y$ both belong to $2T$. Since $d(x', y) \ge \Csep^{-1}$, we have that $2T$ is contained in $100\Csep H'$. This in turn is a $R_j^{-\eta}$-concentrated $(R_j^{-\kappa}/2, k)$-plate, so it must be contained in some $H \in \cH_j$. Hence, $2T \subset H$.

By assumption, we know
$2T \cap B(x, \Csep^{-1})^c \not\subset F(x, aH)$ (where $a = 99\Csep$), so there exists $z \in 2T \cap B(x, \Csep^{-1})^c$ such that $\ell(x, z)$ intersects $\surf(aH)$ at some point $w$. Let $P$ be the $k$-plane through $x$ parallel to the central plane $P_H$ of $H$. Since $x, w, z$ are collinear, we have
\begin{equation*}
    \frac{d(w, P)}{d(z, P)} = \frac{d(w, x)}{d(z, x)}.
\end{equation*}
However, we know that $w \in \surf(aH)$, so $d(w, P_H) = aR_j^{-\kappa}$. Also $d(P, P_H) \le R_j^{-\kappa}$ since $x \in 2T \subset H$. Thus by triangle inequality, we get $d(w, P) \ge (a-1) R_j^{-\kappa}$. While $d(z, P) \le 2 R_j^{-\kappa}$ (since $z \in 2T \subset H$), so $\frac{d(w, P)}{d(z, P)} \ge \frac{a-1}{2}$. On the other hand, $d(w, x) \le 20$ and $d(z, x) \ge \Csep^{-1}$, so $\frac{d(w, x)}{d(z, x)} \le 20\Csep$. Hence, we get $\frac{a-1}{2} \le 20\Csep$, contradiction to $a = 99\Csep$ and $\Csep \ge 1$.


\emph{\textbf{Upper bound of $\mu_2 \times \mu_1 (\Bad^2_j)$.}} We will prove there exists a choice for $a$ such that for all $j, \ell$ satisfying $\ell \le \max(j, j_*)$, we have
\begin{equation}\label{eqn:sum for fixed ell}
    F_{j,\ell} (a) = \sum_{H \in \cH_\ell} \mu_2 \times \mu_1 (\Bad_j^2 (H)) \le R_j^{-1/4}.
\end{equation}
Then given $j$, summing \eqref{eqn:sum for fixed ell} over $\ell \le \max(j, j_*)$ gives \eqref{eqn:bad2 estimate} (use $\max(j, j_*) \lesim \log R_j$). To prove \eqref{eqn:sum for fixed ell}, we fix $j, \ell$ and upper bound the measure of the set of $a$'s for which \eqref{eqn:sum for fixed ell} fails. We would like to apply Markov's inequality, so we compute the expectation of $F_{j,\ell} (a)$ over $a$. Let $P(a) = \frac{1}{\Csep} \one_{[99\Csep, 100\Csep]} \, da$ be a probability measure of $a \in [99\Csep, 100\Csep ]$. Then we have
\begin{align*}
    \int_I F_{j,\ell} (a) \,dP(a)
    = &\int_{E_1}\int_{E_2} \int_I \sum_{H \in \cH_\ell} 1_{\Bad^2_j(H)} (x, y) \,dP(a) d\mu_2(x) d\mu_1(y)\, \\
    \le & \sup_{x \in E_2, y \in E_1} \int_{I} \sum_{H \in \cH_\ell} 1_{\Bad^2_j(H)} (x, y) \,dP(a).
\end{align*}
The following claim shows that it is unlikely that $(x,y)\in \Bad_j^2 (H)$ for a given $H$.

\textbf{Claim 2.} Suppose $(x,y)\in \Bad_j^2 (H)$, where $H \in \cH_{\ell}(x)$ with $\ell\leq \max(j, j_*)$. Let $z_1, z_2$ be the intersections of the line through $x, y$ with $B(0, 10)$. Then for one of $i = 1, 2$, we have $|d(z_i, P_H) - a R_\ell^{-\kappa}| \lesim R_j^{-1/2+\beta}$, where $P_H$ is the central plane of $H$.

\textit{Proof.} By definition of $\Bad_j^2 (H)$, there exists $T\in\T_j$ with $2T$ containing both $x, y$ such that $2T \cap B(x, \Csep^{-1})^c$ is not contained in $\R^d \setminus F(x, aH)$ or $F(x, aH)$.

Fix a large constant $C > 0$. If $d(z_i, P_H) - a R_\ell^{-\kappa} > CR_j^{-1/2+\beta}$ for $i = 1$ or $2$, then we claim $2T \cap B(x, \Csep^{-1})^c \subset \R^d \setminus F(x, aH)$. Note that for any $y' \in 2T \cap B(x, \Csep^{-1})^c$, one of the intersections $z'$ of the line through $x, y'$ with $B(0, 10)$ is contained in $B(z_i, CR_j^{-1/2+\beta})$, and so by triangle inequality, we get $d(z', P_H) > a R_\ell^{-\kappa}$. (This is why the $B(x, \Csep^{-1})^c$ is important: it is not true that for all $y' \in 2T$, one of the intersections $z'$ of the line through $x, y'$ with $B(0, 10)$ is contained in $B(z_i, CR_j^{-1/2+\beta})$. Take $y'$ such that $(y-x) \perp (y'-x)$, for instance.) This shows that $2T \cap B(x, \Csep^{-1})^c \subset \R^d \setminus F(x, aH)$.

If on the other hand $d(z_i, P_H) < a R_\ell^{-\kappa} - CR_j^{-1/2+\beta}$ for both $i = 1, 2$, then a similar argument shows that $2T \cap B(x, \Csep^{-1})^c \subset F(x, aH)$. Thus, we have established the contrapositive of the claim. \qed

Using Claim 2, observe that, for a fixed pair $(x,y)\in E_2\times E_1$ and a fixed $H \in \cH_{\ell}(x)$ with $\ell\leq \max(j, j_*)$ (so $R_\ell \le R_j^2$), we have $(x, y) \in \Bad_j^2 (H)$ for $a$ lying in two intervals each of length $\lesssim R_j^{-1/2+\beta} R_\ell^\kappa \lesssim R_j^{-1/2+\beta+2\kappa}$. Recall that by Lemma \ref{lem:few_large_plates} and $R_\ell \le R_j^2$ we have $|\cH_\ell| \lesssim R_j^{2N\eta}$. Therefore,
\begin{align*}
    \sum_{H \in \cH_\ell} \int_{I} 1_{\Bad^2_j(H)} (x, y) \,dP(a)\lesssim  &\frac{R_j^{2N\eta} R_j^{-1/2+\beta+2\kappa}}{|I|} \\
\sim &R_j^{-1/2+\beta+2N\eta+2\kappa},
\end{align*}
Thus, assuming $\beta, \eta, \kappa$ are small enough and $R_0$ is large enough, by Markov's inequality we have $|\{ a : F_{j,\ell} (a) > R_j^{-1/4} \}| \le R_j^{-1/8}$. By the union bound, \eqref{eqn:sum for fixed ell} fails for some $j, \ell$ satisfying $\ell \le \max(j, j_*)$ only in a set of measure
\begin{align*}
    &\sum_{j=0}^\infty \sum_{\ell \le \max(j, j_*)} R_j^{-1/8} \le \sum_{j=0}^{j_*} \sum_{\ell=0}^{j_*} R_0^{-1/8} + \sum_{\ell=0}^\infty \sum_{j = \ell}^\infty R_j^{-1/8} \\
    \lesim & R_0^{-1/8} \cdot (\log_2 R_0)^2 + \sum_{\ell \ge 0} R_\ell^{-1/8} \lesim R_0^{-1/8} \cdot \left[ (\log_2 R_0)^2 + 1 \right] < 1,
\end{align*}
if $R_0$ is chosen sufficiently large. 
\end{proof}

\section{Refined decoupling and Proposition \ref{mainest2}}\label{sec:prop2.2}
\setcounter{equation}0

In this section, we prove Proposition \ref{mainest2}, which will complete the proof of Theorem \ref{thm: pinned}. This part of the argument proceeds very similarly as \cite[Section 4]{du2021improved} and \cite[Section 5]{guth2020falconer}.

Let $\sigma_r$ be the normalized surface measure on the sphere of radius $r$. The main estimates in the proof of Proposition \ref{mainest2} are the following.

\begin{lemma}\label{lem:big_r}
For any $\alpha > 0$, $r>10R_0$, and $\eps_0$ sufficiently small depending on $\alpha, \eps$:
	$$
	\int_{E_2} | \mu_{1,g}^x * \hsigma_r(x)|^2 d \mu_2(x) \lesssim_{\varepsilon} r^{-\frac{\alpha d}{d+1}+\eps} r^{-(d-1)} \int  | \hmu_1 |^2 \psi_r d \xi+{\rm RapDec}(r),
	$$
where $\psi_r$ is a weight function which is $\sim 1$ on the annulus $r-1 \le |\xi| \le r+1$ and decays off of it.  To be precise, we could take
$$
\psi_r(\xi) = \left( 1 + | r- |\xi|| \right)^{-100}.
$$
\end{lemma}

\begin{lemma}\label{lem:small_r}
    For any $\alpha > 0$, $r > 0$, we have
    \begin{equation*}
        \int_{E_2} |\mu_{1,g}^x * \hsigma_r (x)|^2 \, d\mu_2 (x) \lesim (r+1)^{d-1}r^{-(d-1)} \,.
    \end{equation*}
\end{lemma}

\begin{proof}[Proof of Proposition \ref{mainest2}, given Lemmas \ref{lem:big_r} and \ref{lem:small_r}] 
Note that
    \begin{equation*}
        d_*^x (\mu_{1,g}^x) (t) = t^{d-1} \mu_{1,g}^x * \sigma_t (x).
    \end{equation*}
    Since $\mu_{1,g}^x$ is essentially supported in the $R_0^{-1/2+\beta}$ neighborhood of $E_1$, for $x \in E_2$, we only need to consider $t \sim 1$. Hence, up to a loss of $\text{RapDec}(R_0)$ which is negligible in our argument, we have
    \begin{align*}
        \int \norm{d_*^x (\mu_{1,g}^x)}^2_{L^2} \,d\mu_2 (x) &\lesim \int_0^\infty \int |\mu_{1,g}^x * \sigma_t (x)|^2 \, d\mu_2 (x) t^{d-1} dt \\
        &= \int_0^\infty \int |\mu_{1,g}^x * \hsigma_r (x)|^2 \, d\mu_2 (x) r^{d-1} dr,
    \end{align*}
    where in the second step, we used a limiting process and an $L^2$-identity proved by Liu \cite[Theorem 1.9]{LiuL2}: for any Schwartz function $f$ on $\mathbb R^d, d\geq 2$, and any $x\in \mathbb R^d$,
$$
\int_0^\infty |f*\sigma_t(x)|^2\, t^{d-1}\,dt = 
\int_0^\infty |f*\hsigma_r(x)|^2\, r^{d-1}\,dr.$$

    For $r \le 10R_0$ we use Lemma \ref{lem:small_r}, and for $r > 10R_0$ we use Lemma \ref{lem:big_r}. The small $r$ contribution to $\int \norm{d_*^x (\mu_{1,g}^x)}^2_{L^2} d\mu_2 (x)$ is
    \begin{align*}
        \int_0^{10R_0} (r+1)^{d-1} \, dr \lesim R_0^d.
    \end{align*}
   The large $r$ contribution is (dropping the negligible $\RapDec(r)$ term)
    \begin{align*}
        &\int_{10R_0}^\infty \int_{\R^d} r^{-\frac{\alpha d}{d+1}+\eps} \psi_r (\xi) |\hmu_1 (\xi)|^2 \, d\xi dr \\
        &\lesim \int_{\R^d} |\xi|^{-\frac{\alpha d}{d+1}+\eps} |\hmu_1 (\xi)|^2 \, d\xi.
    \end{align*}
The proof of Proposition \ref{mainest2} is thus complete upon verification of Lemmas \ref{lem:big_r} and \ref{lem:small_r}.
\end{proof}

\begin{proof}[Proof of Lemma \ref{lem:small_r}]
We follow the proof of Proposition 5.3 in \cite{guth2020falconer}, case $r < 10R_0$.
    Since $\mu_2$ is a probability measure, it suffices to upper bound $\sup_x |\mu_{1,g}^x * \hsigma_r (x)|$. Fix $x$ and note that
    \begin{equation*}
        |\mu_{1,g}^x * \hsigma_r (x)|^2 \le \norm{\widehat{\mu_{1,g}^x}}^2_{L^1 (d\sigma_r)} \le \norm{\widehat{\mu_{1,g}^x}}^2_{L^2 (d\sigma_r)}\,.
    \end{equation*}
    Then by the approximately orthogonal argument in \cite[Proof of Proposition 5.3]{guth2020falconer}, we have
    \begin{equation*}
        \norm{\widehat{\mu_{1,g}^x}}^2_{L^2 (d\sigma_r)} \lesim r^{-(d-1)} \int (|\widehat{\mu_1|_{G_0 (x)}}|^2 + |\widehat{\mu_1}|^2) \psi_r d\xi.
    \end{equation*}
     Finally, since $\norm{\hmu_1}_{L^\infty} \le \norm{\mu_1}_{L^1} = 1$, $\norm{\widehat{\mu_1|_{G_0 (x)}}}_{L^\infty} \le \norm{\mu_1|_{G_0 (x)}}_{L^1} \le 1$, and $\int \psi_r \,d\xi \lesim (r+1)^{d-1}$, we get the desired result.
\end{proof}

\subsection{Refined decoupling estimates}
The key ingredient in the proof of Lemma \ref{lem:big_r} is the following refined decoupling theorem, which is derived by applying the $l^2$ decoupling theorem of Bourgain and Demeter \cite{BDdecoupling} at many different scales. 

Here is the setup. Suppose that $S \subset \R^d$ is a compact and strictly convex $C^2$ hypersurface with Gaussian curvature $\sim 1$. For any $\eps>0$, suppose there exists $0<\beta \ll \eps$ satisfying the following.   Suppose that the 1-neighborhood of $R S$ is partitioned into $R^{1/2} \times ... \times R^{1/2} \times 1$ blocks $\theta$.  For each $\theta$, let $\mathbb{T}_\theta$ be a set of finitely overlapping tubes of dimensions $R^{-1/2 + \beta} \times \cdots \times R^{-1/2 + \beta} \times 1$ with long axis perpendicular to $\theta$, and let $\mathbb{T} = \cup_\theta \mathbb{T}_\theta$. Each $T\in \T$ belongs to $\T_{\theta}$ for a single $\theta$, and we let $\theta(T)$ denote
this $\theta$. We say that $f$ is microlocalized to $(T,\theta(T))$ if $f$ is essentially supported in
$2T$ and $\widehat{f}$ is essentially supported in $2\theta(T)$.

\begin{theorem}\label{thm: dec}{\cite[Corollary 4.3]{guth2020falconer}}
Let $p$ be in the range $2 \le p \le \frac{2(d+1)}{d-1}$.  For any $\eps>0$, suppose there exists $0<\beta\ll\eps$ satisfying the following. Let $\mathbb{W} \subset \mathbb{T}$ and suppose that each $T \in \mathbb{W}$ lies in the unit ball.  Let $W = | \mathbb{W}|$.  Suppose that $f = \sum_{T \in \mathbb{W}} f_T$, where $f_T$ is microlocalized to $(T, \theta(T))$.  Suppose that $\| f_T \|_{L^p}$ is $\sim$ constant for each $T \in \mathbb{W}$.   Let $Y$ be a union of $R^{-1/2}$-cubes in the unit ball each of which intersects at most $M$ tubes $T \in \mathbb{W}$.  Then
	$$ \| f \|_{L^p(Y)} \lesssim_\eps R^\eps \left(\frac{M}{W} \right)^{\frac{1}{2} - \frac{1}{p}} \left(\sum_{T \in \mathbb{W}} \| f_T \|_{L^p}^2 \right)^{1/2}. $$
\end{theorem}

The proof of Lemma \ref{lem:big_r} using Theorem \ref{thm: dec} proceeds almost identically as in \cite[Lemma 4.1]{du2021improved}, as the $x$ dependence of the good measure doesn't exist in the regime $r>10 R_0$. We include the proof below for the sake of completeness.

\subsection{Proof of Lemma \ref{lem:big_r}}
Assume $r>10R_0$. By definition, 
\[
\mu_{1,g}^x* \hsigma_r=\sum_{j: R_j \sim r} \sum_{T\in \mathbb{T}_j: T \textrm{ good}} M_T \mu_1 * \hsigma_r+{\rm RapDec}(r).
\]
The key point to notice is that $\mu_{1,g}^x* \hsigma_r$ is independent of $x$.

The contribution of ${\rm RapDec}(r)$ is already taken into account in the statement of Lemma \ref{lem:big_r}. Hence without loss of generality we may ignore the tail ${\rm RapDec}(r)$ in the argument below.

Let $\eta_1$ be a bump function adapted to the unit ball and define
\[
f_T = \eta_1 \left( M_T \mu_1 * \hsigma_r \right).
\]
One can easily verify that $f_T$ is microlocalized to $(T,\theta(T))$.

Let $p=\frac{2(d+1)}{d-1}$. By dyadic pigeonholing, there exists $\lambda>0$ such that
\[
\int |\mu_{1,g}^x*\hsigma_r(x)|^2\,d\mu_2(x)\lesssim \log r \int | f_\lambda (x) |^2 d\mu_2(x),
\]where
\[
f_\lambda=\sum_{T\in \mathbb{W}_\lambda}f_T,\quad \mathbb{W}_\lambda:= \bigcup_{j: R_j\sim r} \Big\{ T\in \T_j: T \text{ good }, \| f_T \|_{L^p} \sim \lambda \Big\}.
\]

Next, we divide the unit ball into $r^{-1/2}$-cubes $q$ and sort them. Denote
\[
\mathcal{Q}_M:=\{ r^{-1/2}\textrm{-cubes } q: q \textrm{ intersects } \sim M \textrm{ tubes } T \in \mathbb{W}_\lambda \}.
\]
Let $Y_{ M} := \bigcup_{q \in \mathcal{Q}_{ M}} q$. Since there are only $\sim \log r$ many choices of $M$, there exists $M$ such that 
\[
\int |\mu_{1,g}^x*\hsigma_r(x)|^2\,d\mu_2(x)\lesssim (\log r)^2 \int_{Y_M} | f_\lambda (x) |^2 d\mu_2(x)\,.
\]

Since $f_\lambda$ only involves good wave packets, by considering the quantity
$$
\sum_{q\in \mathcal{Q}_M} \sum_{T\in \mathbb{W}_\lambda: T\cap q\neq \emptyset} \mu_2(q),
$$
we get
\begin{equation}\label{eqn: count}
    M \mu_2 (\mathcal{N}_{r^{-1/2}}(Y_M) ) \lesssim  | \mathbb{W}_\lambda | r^{-\frac{\alpha}{2}+\eps_0},
\end{equation}
where $\mathcal{N}_{r^{-1/2}}(Y_M)$ is the $r^{-1/2}$-neighborhood of $Y_M$.

The rest of the proof of Lemma \ref{lem:big_r} will follow from Theorem \ref{thm: dec} and estimate (\ref{eqn: count}).

By H\"older's inequality and the observation that $f_\lambda$ has Fourier support in the $1$-neighborhood of the sphere of radius $r$, one has
\[
\int_{Y_M}|f_\lambda(x)|^2\,d\mu_2(x)\lesssim \left(\int_{Y_{M}} | f_\lambda|^p \right)^{2/p} \left(\int_{Y_{M}} |\mu_2 * \eta_{1/r}|^{p/(p-2)} \right)^{1-2/p},
\]
where $\eta_{1/r}$ is a bump function with integral $1$ that is essentially supported on the ball of radius $1/r$.

To bound the second factor, we note that $\eta_{1/r} \sim r^d$ on the ball of radius $1/r$ and rapidly decaying off it. Using the fact that $\mu_2(B(x,t))\lesssim t^\alpha, \forall x\in \R^d, \forall t>0$, we have
$$
\|\mu_2*\eta_{1/r}\|_\infty \lesssim r^{d-\alpha}\,.
$$
Therefore,
\[
\begin{split}
\int_{Y_{M}} |\mu_2 * \eta_{1/r}|^{p/(p-2)} \lesssim & \|\mu_2 * \eta_{1/r}\|_{\infty}^{2/(p-2)} \int_{Y_M} d\mu_2*\eta_{1/r}\\
\lesssim  &r^{2(d-\alpha)/(p-2)}\mu_2(\mathcal{N}_{r^{-1/2}}(Y_M)).
\end{split}
\]By Theorem \ref{thm: dec}, the first factor can be bounded as follows:
\[
\begin{split}
\left(\int_{Y_{M}} | f_\lambda|^p \right)^{2/p}\lessapprox & \left(\frac{M}{\mathbb{W}_\lambda}\right) ^{1-2/p}\sum_{T\in\mathbb{W}_\lambda}\|f_T\|_{L^p}^2\\
\lesssim & \left(\frac{r^{-\frac{\alpha}{2}+\eps_0}}{\mu_2(\mathcal{N}_{r^{-1/2}}(Y_{M}))} \right)^{1-2/p} \sum_{T \in \mathbb{W}_\lambda} \| f_T \|_{L^p}^2,
\end{split}
\]where the second step follows from (\ref{eqn: count}).

Combining the two estimates together, one obtains
\[
\int_{Y_M}|f_\lambda(x)|^2\,d\mu_2(x) \lesssim r^{\frac{2d}{p}-\alpha(\frac 12 +\frac 1p) +O(\eps_0)} \sum_{T \in \mathbb{W}_\lambda} \| f_T \|_{L^p}^2.
\]
Observe that $\|f_T\|_{L^p}$ has the following simple bound:
\[
\begin{split}
\|f_T\|_{L^p}\lesssim & \|f_T\|_{L^\infty}|T|^{1/p}\lesssim \sigma_r(\theta(T))^{1/2}|T|^{1/p} \|\widehat{M_T\mu_1}\|_{L^2(d\sigma_r)}\\
= & r^{-(\frac{1}{2p}+\frac{1}{4})(d-1)+O(\beta)}\|\widehat{M_T\mu_1}\|_{L^2(d\sigma_r)}.
\end{split}
\]Plugging this back into the above formula, one obtains
\[
\begin{split}
\int_{Y_M}|f_\lambda(x)|^2\,d\mu_2(x) \lesssim & r^{\frac{2d}{p}-(\alpha+d-1)(\frac 12 +\frac 1p) +O(\eps_0)} \sum_{T \in \mathbb{W}_\lambda}\|\widehat{M_T\mu_1}\|_{L^2(d\sigma_r)}^2\\
\lesssim & r^{-\frac{\alpha d}{d+1}+\eps} r^{-(d-1)} \int | \hmu_1|^2 \psi_r \,d \xi,
\end{split}
\]where $p=2(d+1)/(d-1)$ and we have used orthogonality and chosen $\beta \ll \eps_0 \ll \eps$. The proof of Lemma \ref{lem:big_r} and hence Proposition \ref{mainest2} is complete.

\section{Proof of Theorem \ref{thm:distance_hausdorff}}\label{sec:hausdorff}
We will prove Theorem \ref{thm:distance_hausdorff} following the approach in \cite{liu2020hausdorff}. We shall use the following criteria to determine the Hausdorff dimension of pinned distance sets. 

\begin{lemma}\label{strategy}{\cite[Lemma 3.1]{liu2020hausdorff}}
Given a compact set $E\subset\R^d$, $x\in\R^d$ and a probability measure $\mu_E$ on $E$. Suppose there exist $\tau\in (0,1]$, $K\in\Z_+$, $\beta>0$ such that
	$$\mu_E(\{y: |y-x|\in D_k\})<2^{-k\beta}$$ 
	for any
	$$D_k=\bigcup_{j=1}^M I_j,$$
	where $k>K$, $M\leq 2^{k\tau}$ are arbitrary integers and each $I_j$ is an arbitrary interval of length $\approx 2^{-k}$. Then
	$$\dim_H(\Delta_x(E))\geq\tau. $$
\end{lemma}

The next proposition is a key step in the proof of Theorem \ref{thm:distance_hausdorff}, which can be viewed as a discretized variant of it.

\begin{prop}\label{prop:non-concentration}
    Let $d\geq 2, 1 \le k \le d-1$, $k-1 < \alpha \le k$, and $\tau < \min(f(\alpha), 1)$. There exists $\beta > 0$ depending on $\tau, \alpha, k$ such that the following holds for sufficiently small $\delta < \delta_0 (\tau, \alpha, k)$. Let $\mu_1, \mu_2$ be $\alpha$-dimensional measures with $\sim 1$ separation and constant $C_\alpha$ supported on $E_1, E_2$ respectively. Then there exists a set $F_\delta \subset E_2$ with $\mu_2 (F_\delta) \lesim \delta^{\beta^2}$ such that for all $x \in E_2 \setminus F_\delta$, there exists a set $W(x)$ that is contained within some $(2\delta^{\beta^2}, k)$-plate such that
    \begin{equation*}
        \mu_1 (\{ y : |x-y| \in J \} \setminus W(x)) \le \delta^{\beta^2/2},
    \end{equation*}
    where $J$ is any union of $\le \delta^{-\tau}$ many intervals each of length $\sim \delta$.
\end{prop}

Next we prove Proposition \ref{prop:non-concentration}. The proof reproduces the argument in \cite[Section 4]{liu2020hausdorff} with some minor simplifications.

\begin{proof}
    First, instead of working with $\mu_1$, we will use a mollified version that removes the high frequency contributions. Let $\phi \in C_0^\infty (\R^d)$ be supported on $B(0, 1)$ and satisfy $\phi \ge 0$, $\int \phi = 1$, and $\phi \ge 1$ on $B(0, \frac{1}{2})$. This $\phi$ will be fixed for the rest of the proof (in particular, it does not depend on $\delta$, and subsequent implicit constants may depend on $\phi$). Let $\phi_\delta (\cdot) = \delta^{-d} \phi(\delta^{-1} \cdot)$ and $\mu_1^\delta = \mu_1 * \phi_\delta$. The crucial point is that $\mu_1^{\delta}$ is supported in a $\delta$-neighborhood of the support of $\mu_1$ and in fact serves as a good approximation for $\mu_1$ down to scale $\delta$, but $\mu_1^\delta$ is rapidly decaying at frequencies much larger than $\delta^{-1}$.

    Fix a small $\eps > 0$. We apply Proposition \ref{mainest1} with $R_0 = \delta^{-\beta}$ and
    the measure $\mu_1^\delta$, which is still an $\alpha$-dimensional measure with constant comparable to $C_\alpha$ (independent of $\delta$).
    (We make $\delta$ sufficiently small to ensure that $R_0$ is sufficiently large.) Then there is a subset $E_2' \subset E_2$ so that $\mu_2(E_2') \ge 1 - \delta^{\beta^2}$ and for each $x \in E_2'$, there exists a set $G(x) \subset \R^d$ where $B^d(0,10) \setminus G(x)$ is contained within some $(\delta^{\beta^2}, k)$-plate $H(x)$ such that
\begin{equation*} 
        \norm{d_*^x (\mu^\delta_1|_{G(x)}) - d_*^x (\mu_{1,g}^{\delta, x})}_{L^1} \le \delta^{\beta^2}.
\end{equation*}
We will define $W(x) = H(x)^{(\delta)}$, which satisfies the condition for $W(x)$. Let
    \begin{equation*}
        \cJ^\tau_\delta = \left\{ \bigcup_{j=1}^M I_j : M \le \delta^{-\tau}, \text{ each } I_j \text{ is an open interval of length} \sim \delta  \right\}.
    \end{equation*}
Let $F'$ be the set of points $x \in E_2'$ such that
    \begin{equation*}
        \sup_{J \in \cJ^\tau_\delta} \int_{J^{(\delta)}} d_*^x (\mu_1^{\delta}|_{G(x)}) (t) \, dt \ge \delta^{\beta^2/2}.
    \end{equation*}
Now, define $F_\delta := F' \cup (E_2 \setminus E_2')$. Then for any $x\in E_2 \setminus F_\delta = E_2' \setminus F'$, we can claim the following.

    \textbf{Claim.} For all $x \in E_2' \setminus F'$ and $J \in \cJ_\delta^\tau$, we have
    \begin{equation*}
        \mu_1 (\{ y : |x-y| \in J\} \setminus W(x)) \le \delta^{\beta^2/2}.
    \end{equation*}

    \textit{Proof of Claim.} Note that if $y \notin W(x)$, then $B(y, \delta) \subset G(x)$. For $x \in E_2' \setminus F'$ and $J \in \cJ_\delta^\tau$, we have
    \begin{align*}
        \delta^{\beta^2/2} &\ge \int_{|x-z| \in J^{(\delta)}} \mu_1^{\delta}|_{G(x)} (z) \, dz \\
        &= \delta^{-d} \iint_{|x-z| \in J^{(\delta)}, z \in G(x)} \phi(\delta^{-1} (z-y)) d\mu_1 (y) dz \\
        &\ge \delta^{-d} \iint_{|x-y| \in J, |y-z| \le \delta, y \notin W(x)} \phi(\delta^{-1} (z-y)) d\mu_1 (y) dz \\
        &\ge \int_{|x-y| \in J, y \notin W(x)} d\mu_1 (y) \int_{B(0, 1)} \phi(u) \, du \\
        &= \mu_1 (\{ y : |x-y| \in J \} \setminus W(x)). \qed
    \end{align*}
    

    Recall that $\mu_2(E_2\setminus E_2')\leq \delta^{\beta^2}$. So it remains to show $\mu_2 (F') \lesssim \delta^{\beta^2}$ (assuming good choice for $\beta, \eps$). For $x \in F'$, we have
    \begin{align*}
        \sup_{J \in \cJ^\tau_\delta} \int_{J^{(\delta)}} & |d_*^x (\mu_{1,g}^{\delta,x}) (t)| \, dt \\
        &\ge \sup_{J \in \cJ^\tau_\delta} \int_{J^{(\delta)}} d_*^x (\mu_1^\delta|_{G(x)}) (t) \, dt - \norm{d_*^x (\mu_1^\delta|_{G(x)}) - d_*^x (\mu_{1,g}^{\delta,x})}_{L^1} \\
        &\ge \delta^{\beta^2/2} - \delta^{\beta^2} \ge \delta^{\beta^2}.
    \end{align*}
Then by Cauchy-Schwarz, we have for $x \in F'$,
\begin{equation*}
    \left( \sup_{J \in \cJ_\delta^\tau} |J^{\delta}|^{\frac{1}{2}} \right) \left( \int |d_*^x (\mu_{1,g}^{\delta,x})(t)|^2 dt \right)^{\frac{1}{2}} \ge \sup_{J \in \cJ^\tau_\delta} \int_{J^{(\delta)}} |d_*^x (\mu_{1,g}^{\delta,x}) (t)| \, dt \ge \delta^{\beta^2}.
\end{equation*}
For $J \in \cJ_\delta^\tau$, $J$ and $J^{(\delta)}$ can both be covered by $\lesim \delta^{-\tau}$ many intervals each of length $\sim \delta$, so $\sup_{J \in \cJ_\delta^\tau} |J^{\delta}|^{1/2} \lesim \delta^{(1-\tau)/2}$. Thus, for $x \in F'$,
\begin{equation*}
    \norm{d_*^x (\mu_{1,g}^{\delta,x})}^2_{L^2} \ge \delta^{2\beta^2-(1-\tau)}q.
\end{equation*}
Integrate over $F'$ and apply Proposition \ref{mainest2} to get
\begin{align*}
    \delta^{2\beta^2-(1-\tau)} \mu_2 (F') &\le \int \norm{d_*^x (\mu_{1,g}^{\delta,x})}^2_{L^2} d\mu_2 (x) \\
            &\le \int |\widehat{\mu_1^{\delta}} (\xi)|^2 |\xi|^{-\frac{\alpha d}{d+1} + \eps} \, d\xi + \delta^{-d\beta} + \RapDec(\delta) \\
            &\le \int |\widehat{\mu_1} (\xi)|^2 |\hphi(\delta \xi)|^2 |\xi|^{-\frac{\alpha d}{d+1} + \eps} \, d\xi + \delta^{-d\beta} + \RapDec(\delta) \\
            &\le \int_{|\xi| \le \delta^{-1-\beta}} |\widehat{\mu_1} (\xi)|^2 |\xi|^{-\frac{\alpha d}{d+1} + \eps} \, d\xi + \delta^{-d\beta} + \RapDec(\delta)\,.
\end{align*}
Since $\tau < 1$, the condition $|\xi| \le \delta^{-1-\beta}$, which implies $|\xi|^{1-\beta} \le \delta^{-(1+\beta)(1-\beta)} \le \delta^{-1}$, gives us
\begin{equation*}
    \delta^{-3\beta^2+(1-\tau)} \le |\xi|^{(3\beta^2-(1-\tau))(1-\beta)} = |\xi|^{-(1-\tau)+O(\beta)}.
\end{equation*}
Thus,
\begin{align*}
    &\mu_2 (F') \\
    \le &\delta^{-2\beta^2+(1-\tau)} \int_{|\xi| \le \delta^{-1-\beta}} |\widehat{\mu_1} (\xi)|^2 |\xi|^{-\frac{\alpha d}{d+1} + \eps} \, d\xi + \delta^{-d\beta-2\beta^2+(1-\tau)} + \RapDec(\delta) \\
        \le &\delta^{\beta^2} \int |\widehat{\mu_1} (\xi)|^2 |\xi|^{-\frac{\alpha d}{d+1} + \eps - (1-\tau)+O(\beta)} \, d\xi + \delta^{-d\beta-2\beta^2+(1-\tau)} + \RapDec(\delta)\,.
\end{align*}
Finally, since $\tau < f(\alpha) = \alpha \cdot \frac{2d+1}{d+1} - (d-1)$ and $\tau < 1$, we may choose small $\eps$ and $\beta$ such that
\begin{gather*}
    -d\beta-2\beta^2 + (1-\tau) > \beta^2, \\
    -\frac{\alpha d}{d+1} + \eps - (1-\tau) + O(\beta) < -d + \alpha.
\end{gather*}
This guarantees the energy integral
\begin{equation*}
    \int |\widehat{\mu_1} (\xi)|^2 |\xi|^{-\frac{\alpha d}{d+1} + \eps - (1-\tau) + O(\beta)}\, d\xi
\end{equation*}
to be finite (in fact, $\lesim C_\mu$) and so $\mu_2 (F') \lesim \delta^{\beta^2}$.
\end{proof}

We will also need a result of Shmerkin to prove Theorem \ref{thm:distance_hausdorff}.

\begin{theorem} \cite[Theorem 6.3 and Theorem B.1]{shmerkin2022non} \label{thm: shmerkin}
    Fix $1\le k \le d-1$, $c>0$. Given $\kappa_1,\kappa_2>0$, there is $\gamma>0$ (depending continuously on $\kappa_1,\kappa_2$) such that the following holds.

Let $\mu,\nu$ be probability measures on $B^d(0,1)$ satisfying decay conditions
\begin{align*}
\mu(V) &\le C_\mu r^{\kappa_1},\\
\nu(V) &\le C_\nu r^{\kappa_2},
\end{align*}
for any $(r, k-1)$-plate $V$, and $0<r\leq 1$. Suppose $\nu$ gives zero mass to every $k$-dimensional affine plane.
Then  for all $x$ in a set $E$ of $\mu$-measure $\ge 1-c$ there is a set $K(x)$ with $\nu(K(x))\ge 1-c$ such that
\begin{equation} \label{eq:spherical-proj-non-concentration}
\nu(W\cap K(x)) \le r^\gamma,
\end{equation}
for any $r\in (0,r_0]$ and any $(r,k)$-plate $W$ passing through $x$, where $r_0>0$ depends only on $d, \mu, C_\nu, \kappa_2, c$.

Finally, the set $\{ (x,y) : x\in E, y\in K(x) \}$ is compact.
\end{theorem}

We now prove the following theorem that implies the second claim of Theorem \ref{thm:distance_hausdorff} (that $\sup_{x \in E} \dim_H (\Delta_x (E)) \ge \min(f(\alpha), 1)$). (Specifically, apply Theorem \ref{thm:distance_hausdorff'} with $\alpha-\eps$ for any $\eps > 0$.) The first and third claims then follow by applying the second claim of Theorem \ref{thm:distance_hausdorff} to the set $\{ x \in E : \dim_H (\Delta_x (E)) \le \min(f(\alpha), 1)-\eps\}$ and taking a sequence of $\eps_n \to 0$.

\begin{theorem}\label{thm:distance_hausdorff'}
    Let $0<\alpha\leq d-1$. Suppose $E_1, E_2 \subset B^d(0,1)$ are separated by $\sim 1$, and each of them has positive $\alpha$-dimensional Hausdorff measure. Then there exists $x \in E_1 \cup E_2$ with $\dim_H (\Delta_x (E_1)) \ge \min(f(\alpha), 1)$.
\end{theorem}

\begin{proof}
Let $1 \le k \le d-1$, $k-1 < \alpha \le k$. Let $\mu_1, \mu_2$ be $\alpha$-dimensional measures supported on $E_1, E_2 $ respectively.
Suppose $\mu_1$ gives nonzero mass to some $k$-dimensional affine plane $H$.  We have  three possible cases:
\begin{itemize}
    \item If $k \ge 3$, then $\alpha > k-1 \ge \frac{k+1}{2}$, so by \cite{peres2000smoothness}, there exists $x \in E_1$ such that $|\Delta_x (E_1)| > 0$.
    
    \item If $k = 2$, then by \cite[Theorem 1.1]{liu2020hausdorff} we have  $\dim_H (\Delta_x (E_1)) \ge \min(f(\alpha), 1)$ for some $x\in E_1$. 

    \item If $k = 1$, then for all $x\in E_1 \cap H$, we have $\dim_H (\Delta_x (E_1)) \ge \dim_H (E_1 \cap H) \ge \alpha > f(\alpha)=\min(f(\alpha), 1)$. 
\end{itemize}
Now assume $\mu_1$ gives zero mass to every $k$-dimensional affine plane. 
Also, note that $\mu_1$ and $\mu_2$ satisfy decay conditions
$$
\mu_{i}(V) \lesssim C_{\mu_i} r^{\alpha-(k-1)}, \quad i=1,2,
$$
for any $(r, k-1)$-plate $V$, and $0<r\leq 1$.
Then we can use Theorem \ref{thm: shmerkin}
to find $r_0, \gamma > 0$ and a set $E_2' \subset E_2$ with $\mu_2 (E_2') \ge \frac{1}{2}$ such that for any $x \in E_2'$, there exists $K(x) \subset E_1$ with $\mu_1 (K(x)) \ge \frac{1}{2}$ such that for any $(r, k)$-plate $H$ containing $x$ with $r \le r_0$,
    \begin{equation*}
        \mu_1 (K(x) \cap H) \le r^\gamma.
    \end{equation*}
    Additionally, the set $\{ (x, y) : x \in E_2', y \in K(x) \}$ is compact, which in particular means that $K(x)$ is compact for all $x \in E_2'$.

    Fix $0 < \tau < \min(f(\alpha), 1)$. We apply Proposition \ref{prop:non-concentration} at all sufficiently small dyadic scales $\delta$. By the Borel-Cantelli lemma, a.e. $x \in E_2'$ lie in finitely many of the $F_\delta$. For such $x$, we have for all sufficiently small $\delta > 0$ and any $J$ which is a union of $\leq \delta^{-\tau}$ many intervals each of length $\sim \delta$,
    \begin{align*}
        \mu_1 (\{y : |x-y|& \in J \} \cap K(x)) \\
        &\le \mu_1 (\{y : |x-y| \in J \} \setminus W(x)) + \mu_1 (K(x) \cap W(x)) \\
        &\lesim \delta^{\beta^2/2} + \delta^{\gamma \beta^2}.
    \end{align*}
    Hence, by Lemma \ref{strategy} applied to the restricted measure $\mu_1|_{K(x)}$, we see that $\dim_H (\Delta_x (E_1)) \ge \tau$ for a.e. $x \in E_2'$. Taking a sequence of $\tau \to \min(f(\alpha), 1)$, we see that $\dim_H (\Delta_x (E_1)) \ge \min(f(\alpha), 1)$ for a.e. $x \in E_2'$.
\end{proof}

\section{Other norms and connections with Erd\H{o}s distance problem} \label{sec:othernorms}
Theorem \ref{thm: pinned} also extends to more general norms. For a symmetric convex body $K$ in $\R^d$, let $\|\cdot\|_K$ be the norm with unit ball $K$. 

\begin{theorem}\label{thm: pinned_general}
Let $d\geq 3$. Let $K$ be a symmetric convex body in $\R^d$ whose boundary is $C^\infty$ smooth and has strictly positive curvature.
Let $E\subset\mathbb{R}^d$ be a compact set. Suppose that $\dim_H (E)>\frac{d}{2}+\frac{1}{4}-\frac{1}{8d+4}$. Then there is a point $x\in E$ such that the pinned distance set $\Delta_{K,x}(E)$ has positive Lebesgue measure, where
$$
\Delta_{K,x}(E):=\{\|x-y\|_K:\, y\in E\}.
$$
\end{theorem}

The argument for this generalization is similar to that in \cite{guth2020falconer}. Indeed, the definition of bad tubes and heavy plates depends only on the geometry of $\R^d$ and not the specific norm involved (note that the $r$-neighborhood of a set $A$ is still defined via Euclidean metric, not the new norm's metric). The change of norm only affects the conversion from geometry to analysis, as manifested in Lemma \ref{lem:types of tube}(a) and our use of Liu's $L^2$-identity \cite[Theorem 1.9]{LiuL2} in the proof of Proposition \ref{mainest2}. These considerations were already done in \cite{guth2020falconer} (see also \cite{du2021improved}); we omit the details.

As discussed in \cite{guth2020falconer, du2021improved}, one can go from Falconer-type results to Erd\H os-type results. 

\begin{definition} \label{sadaptable} Let $P$ be a set of $N$ points contained in ${[0,1]}^d$. Define the measure
\begin{equation} \label{pizdatayamera} d \mu^s_P(x)=N^{-1} \cdot N^{\frac{d}{s}} \cdot \sum_{p \in P} \chi_B(N^{\frac{1}{s}}(x-p))\, dx, \end{equation} where $\chi_B$ is the indicator function of the ball of radius $1$ centered at the origin. We say that $P$ is \emph{$s$-adaptable} if there exists $C$ independent of $N$ such that
\begin{equation} \label{sadaptenergy} I_s(\mu_P)=\int \int {|x-y|}^{-s} \,d\mu^s_P(x) \,d\mu^s_P(y) \leq C. \end{equation}
\end{definition}

It is not difficult to check that if the points in set $P$ are separated by distance $cN^{-1/s}$, then (\ref{sadaptenergy}) is equivalent to the condition
\begin{equation} \label{discreteenergy} \frac{1}{N^2} \sum_{p \not=p'} {|p-p'|}^{-s} \leq C, \end{equation}where the exact value of $C$ may be different from line to line. In dimension $d$, it is also easy to check that if the distance between any two points of $P$ is $\gtrsim N^{-1/d}$, then (\ref{discreteenergy}) holds for any $s \in [0,d)$, and hence $P$ is $s$-adaptable.

Using the same argument as in \cite{du2021improved}, from Theorem \ref{thm: pinned_general} we get the following Erd\H os-type result.
\begin{prop}
Let $d\geq 3$. Let $K$ be a symmetric convex body in $\R^d$ whose boundary is $C^\infty$ smooth and has strictly positive curvature. Let $P$ be a set of $N$ points contained in $[0,1]^d$. 

(a). If the distance between any two points of $P$ is $\gtrsim N^{-1/d}$, then there exists $x\in P$ such that
$$
\left| \Delta_{K,x}(P) \right| \gtrapprox N^{\frac{1}{\frac{d}{2}+\frac{1}{4}-\frac{1}{8d+4}}}=N^{\frac{2d+1}{d(d+1)}}.
$$

(b). More generally, if $P$ is $s_n$-adaptable for a decreasing sequence $(s_n)_{n=1}^\infty$ converging to $\frac{d}{2}+\frac{1}{4}-\frac{1}{8d+4}$, then there exists $x\in P$ such that
$$
\left| \Delta_{K,x}(P) \right| \gtrapprox N^{\frac{2d+1}{d(d+1)}}.
$$
\end{prop}

\bibliographystyle{plain}
\bibliography{main}

\vspace{1cm}
	
\noindent Xiumin Du, Northwestern University, \textit{xdu@northwestern.edu}\\

\noindent Yumeng Ou, University of Pennsylvania, \textit{yumengou@sas.upenn.edu}\\

\noindent Kevin Ren, Princeton University, \textit{kevinren@princeton.edu}\\

\noindent Ruixiang Zhang, UC Berkeley, \textit{ruixiang@berkeley.edu}

\end{document}